\documentclass[oneside, a4paper,11pt,reqno]{amsart}
\usepackage{color}
\usepackage{amssymb,amsmath,amsthm}
\usepackage{graphics,graphicx}
\usepackage[T1]{fontenc}

\newtheorem{hypo}{Hypothesis}

\newtheorem{prop}[hypo]{Proposition}

\newtheorem{thm}[hypo]{Theorem}

\newtheorem{lem}[hypo]{Lemma}

\newtheorem{defi}[hypo]{Definition}

\newtheorem{rqe}[hypo]{Remark}

\newtheorem{coro}[hypo]{Corollary}

\def\C{\mathcal{C}}

\def\S{\mathcal{S}}

\def\F{\mathcal{F}}

\title{On the degree of caustics by reflection}
\date\today

\author{Alfrederic Josse}
\address{Universit\'e Europ\'eenne de Bretagne, Universit\'e de Brest,
Laboratoire de Math\'ematiques de Bretagne Atlantique, UMR CNRS 6205, 29238 Brest cedex, France}
\email{alfrederic.josse@univ-brest.fr}

\author{Fran\c{c}oise P\`ene}
\address{Universit\'e Europ\'eenne de Bretagne, Universit\'e de Brest,
Laboratoire de Math\'ematiques de Bretagne Atlantique, UMR CNRS 6205, 29238 Brest cedex, France}
\email{francoise.pene@univ-brest.fr}

\keywords{caustic, degree, polar, intersection number, pro-branch,
Pl\"ucker formula\\
Fran\c{c}oise P\`ene is partially supported by the franch ANR project GEODE}

\begin{document}

\begin{abstract}
Given $\mathcal S\in \mathbb P^2$ and an algebraic curve $\mathcal C$ of $\mathbb P^2$
(with any type of singularities), 
we consider the lines $\mathcal R_m$ 
got by reflection of lines  $(\mathcal Sm)$ ($m\in\mathcal C$) 
on $\mathcal C$. The caustic by reflection
$\Sigma_{\mathcal S}(\mathcal C)$ is defined as the Zariski closure of the envelope of the 
reflected lines $\mathcal R_m$.
We identify this caustic with the Zariski closure of $\Phi(\mathcal C)$,
where $\Phi$ is some rational map.
We use this approach to give general and explicit formulas for the degree (with multiplicity) of 
every caustic by reflection. Our formulas
are expressed in terms of intersection numbers of the initial curve $\mathcal C$ (or 
of its branches). 
Our method is based on a fundamental lemma for rational map thanks to the
notion of $\Phi$-polar and on
computation of intersection numbers. In particular, we use precise estimates 
related to the intersection numbers of $\mathcal C$ with its polar at any point and
to the intersection numbers of $\mathcal C$ with its hessian determinant. These computations are
linked with generalized Pl\"ucker formulas for the class and for the number of inflection points
of $\mathcal C$.
\end{abstract}
\maketitle
\section*{Introduction}
Von Tschirnhausen was the first to consider 
the caustic by reflection as the envelope of reflected rays from a point $\mathcal S$ 
on a mirror curve $\mathcal C$ (see \cite{Tschirnhausen}). 
Many mathematicians have studied individually different 
caustics. In \cite{Quetelet,Dandelin}, when $\S$ is at finite distance,
Quetelet and Dandelin showed that
the caustic is the evolute of the $\mathcal S$-centered homothety
(with ratio $2$) of the pedal 
from $\mathcal S$ of $\mathcal C$, i.e. the evolute of the orthotomic of $\mathcal C$
with respect to $\mathcal S$. This decomposition has also been used
in a modern approach by \cite{BGG1,BGG2,BG} to study the source genericity 
(in the real case).

In \cite{Chasles},
Chasles got a formula (in generic but restrictive cases) 
for the class of the caustic in terms of the degree and of the class of $\mathcal C$. 
In \cite[p. 137, 154]{Salmon-Cayley}, Salmon and Cayley establish formulas, at a more general level, 
for the class and the degree of the evolute and pedal curves. 
The formulas of Salmon and Cayley apply only to curves having no singularities other than ordinary 
nodes and cups \cite[p. 10]{Salmon-Cayley}.

Apparently thanks to these last results, in \cite{Brocard-Lemoine}, 
Brocard and Lemoyne 
gave, without any proof, formulas for the degree and the class of caustics by reflection,
for $\mathcal S$ is at finite distance and for algebraic curve $\mathcal C$
admitting no other singularities than ordinary nodes and cusps.
The formulas of Brocard and Lemoyne are not satisfactory. First no proof is given.
Second, the direct composition of the formulas got by Salmon and Cayley for evolute and pedal 
curves is not correct since the pedal curve of a curve having no singularities other than
ordinary nodes and cups is not necessarily a curve satisfying the same properties.
For example, the pedal curve of the rational cubic $V(y^2z-x^3)$ from $[4:0:1]$ is
a quartic curve with a triple ordinary point.

More recently, a study of the evolute has been done by Fantechi in \cite{Fantechi},
including necessary and sufficient condition for the birationality of the evolute map
and a description of the number and type of the singularities of the general evolute.
This work has been extended in higher dimension by Trifogli \cite{Trifogli}, Catanese
and Trifogli \cite{CataneseTrifogli} giving, in particular,
formulas for degrees of focal loci of smooth algebraic curves.

Our aim is here to give formulas for the degree (with multiplicity) 
of the caustic by reflection
for any light point $\mathcal S$ (including the case when $\mathcal S$ is on the
infinite line) and any algebraic curve $\mathcal C$ (without any restriction neither on
the singularity points nor on the flex points). 
We express the degree (with multiplicity) of $\Sigma_{\mathcal S}(\mathcal C)$ in terms of intersection numbers
of the initial curve $\mathcal C$.
Our proofs use the notion of pro-branches (also called partial branches) considered by
Halphen \cite{Halphen} and more recently by Wall \cite{Wall,Wall2}.

Given an algebraic curve $\mathcal{C}$ in the
euclidean affine plane $\mathcal{E}_{2}$ ($\mathcal{C}$ is called mirror
curve) and given a light position $\mathcal S$ (in $\mathcal{E}_{2}$ or at infinity),
the caustic by reflection $\Sigma _{\mathcal S}(\mathcal{C})$ is the Zariski closure of the envelope of the
reflected lines $\{R_{m};m\in \mathcal{C}\}$ where, for every $m\in \mathcal{
C}$, the reflected line $R_{m}$ at $m$ is the line containing $m$ and such that the
tangent line ${\mathcal T}_{m}\mathcal C$ to $\mathcal{C}$ at $m$ is the bissectrix of the
incident line $(Sm)$ and of $R_{m}$.

\begin{center}
\includegraphics[scale=0.43]{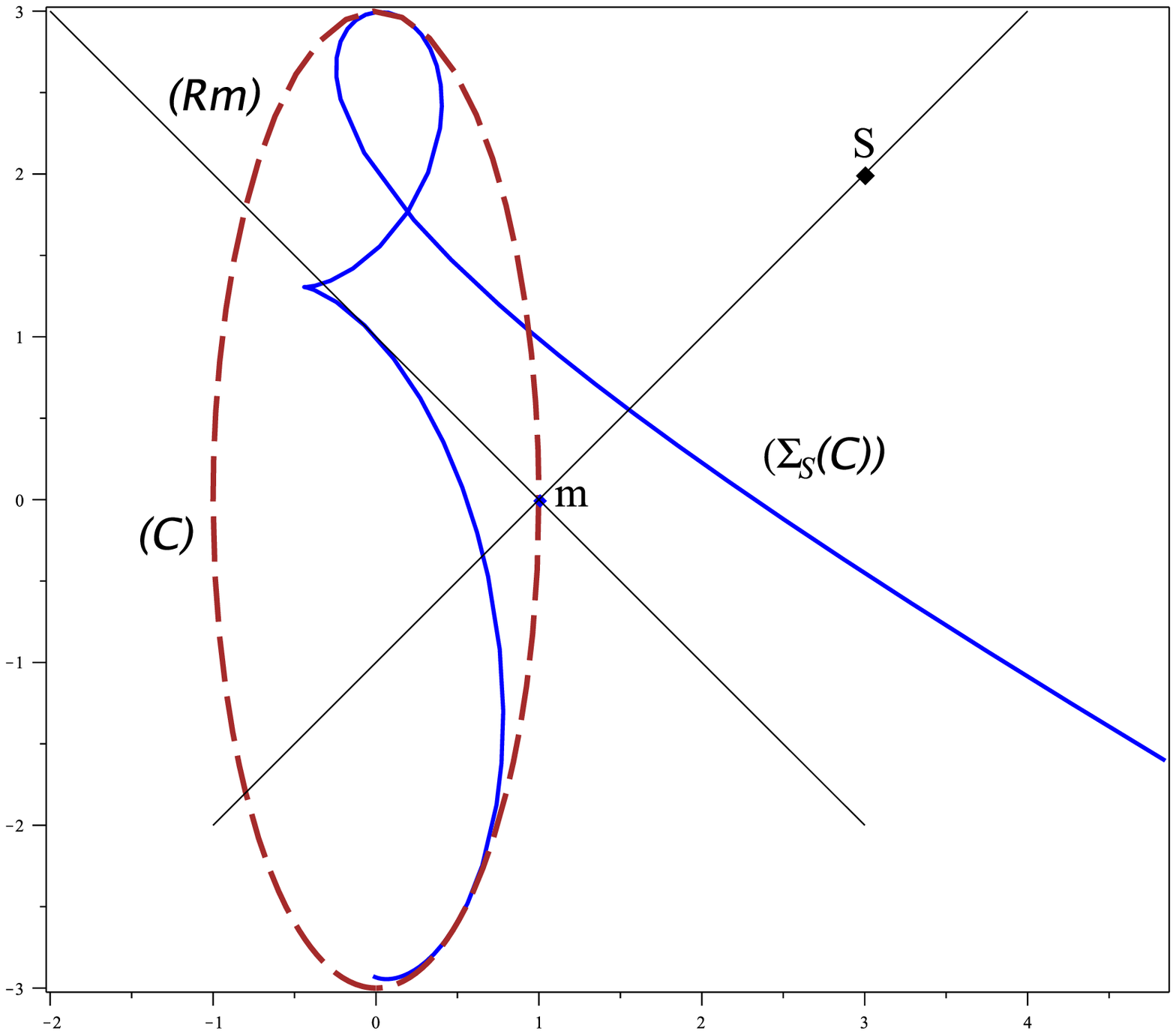}
\end{center}

The notion of caustic by reflection $\Sigma_{\mathcal S}(\mathcal{C})$ is easily extendible
to the complex projective case for an irreducible algebraic curve $\mathcal{C}=V(F)$ 
of $\mathbb{P}^{2}:={\mathbb P}^{2}(\mathbb{C})$
with $F\in{\mathbb C}[x,y,z]$ a homogeneous polynomial
of degree $d$ and a light position $\mathcal S=[x_{0}:y_{0}:z_{0}]\in {\mathbb P}^{2}$.
It will be also useful to consider $S:=(x_0,y_0,z_0)\in{\mathbb C}^3\setminus\{0\}$.

\subsection*{Plan of the paper}
The paper is organized as follows. 

In section \ref{main}, we present our main results and illustrate them with an example.

In section \ref{defcaustique}, we introduce the notion
of reflected lines and use it to define the caustic by reflection.

In section \ref{pptesphi}, we study the properties of our rational map $\Phi_{F,S}$
(link with the caustic $\Sigma_{\mathcal S}(\mathcal C)$, base points,etc.).

In section \ref{sectionlemmefond}, for any rational map $\varphi:\mathbb P^p\rightarrow\mathbb P^q$
and any irreducible algebraic curve $\mathcal C$ of $\mathbb P^p$,
we introduce the notion of $\varphi$-polar
$\mathcal P_{\varphi,a}$ and give a general fundamental lemma expressing the degree of 
$\overline{\varphi(\mathcal C)}$ 
in terms of intersection numbers of $\mathcal C$ with 
$\mathcal P_{\varphi,a}$ at the Base points of $\varphi$ on $\mathcal C$. 
Thanks to this result, the computation of the degree
of the caustic by reflection $\Sigma_{\mathcal S}(\mathcal C)$ is related
to the intersection numbers of the curve $\mathcal C$ 
with its polar curves $\delta_{\mathcal P}F$ with respect to $\mathcal P$ and with
$V(H_F)$ ($H_F$ being the Hessian determinant of $F$). 

In section \ref{rappels} completed with the appendix, 
we recall some facts on intersection numbers 
and use it to establish the precise computations we need. This section contains also
generalized Pl\"ucker formulas for the class and for the number of inflection points
of $\mathcal C$.

In section \ref{preuve}, we prove our general Theorems \ref{degrecaustique1} and 
\ref{degrecaustique2} on the degree of caustics by reflection.
In section \ref{preuvecoro}, we prove our Corollary \ref{CORO1}.

\section{Main results}\label{main}
Throughout the paper, we will write $\ell_\infty$ the infinite line of $\mathbb P^2$ and
$\Pi :{\mathbb C}^3\setminus\{0\}\rightarrow{\mathbb P}^2$ the canonical projection.
As usual, we denote by $F_x,F_y,F_z$ the partial derivatives of $F$ and by
$F_{xx},F_{xy},F_{xz},F_{yy},F_{yz},F_{zz}$ its second order partial derivatives.

We recall that when $d=1$, $\Sigma_{\mathcal S}(\mathcal{C})$ is well defined as soon as
$\mathcal C$ contains neither $\mathcal S$, nor $\mathcal I$, nor $\mathcal J$ and 
we have $\Sigma_{\mathcal S}(\mathcal{C})=\{\mathcal S_1\}$ 
(where $\mathcal S_1$ corresponds to the euclidean symetric point\footnote{
If $F(x,y,z)=ax+by+cz$,
we have $\mathcal S_1=\Pi \left((a^2+b^2)S-2(x_0a+y_0b+cz_0)\left(
    \begin{array}{c}a\\b\\0\end{array}\right)\right)$.} of
$\mathcal S$ with respect to line $\mathcal C$). 

The aim of this paper is to give an effective way to compute the degree of
$\Sigma _{\mathcal S}(\mathcal{C})$ when $d\ge 2$. To this end, we define a rational map
$\Phi_{F,S}$ (with $S=(x_0,y_0,z_0)$). This maps is given
by $\tilde\Phi_{F,S}:{\mathbb C}^3\rightarrow{\mathbb C}^3$ defined by~:
\begin{equation}\label{definitionPhi}
\tilde\Phi_{F,S}:=-\frac{2H_FN_S}{(d-1)^2}\cdot\textrm{Id}+\Delta_SF\cdot
  \left(\begin{array}{c}  F_y^2x_0-F_x^2x_0-2F_xF_yy_0-2F_xF_zz_0\\
F_x^2y_0-F_y^2y_0-2F_xF_yx_0-2F_yF_zz_0\\
z_0(F_x^2+F_y^2)\end{array}\right),
\end{equation}
with $H_F$ the hessian determinant of $F$, i.e.
$$H_F:=F_{xx}F_{yy}F_{zz}-F_{xx}F_{yz}^2+2F_{xy}F_{yz}F_{xz}-F_{xy}^2F_{zz}-F_{yy}F_{xz}^2,$$
with
$$N_S(x,y,z):=(xz_0-x_0z)^2+(yz_0-y_0z)^2 ,$$
and with
$$\forall P=(x_P,y_P,z_P)\in{\mathbb C}^3\setminus\{0\},\ \ 
       \Delta_PF:=DF(\cdot)(P)=x_PF_x+y_PF_y+z_PF_z.$$
Let us recall that $V(\Delta_PF)$ is the polar $\delta_{\Pi(P)}({\mathcal C})$ of
$\mathcal C$ with respect to $\Pi(P)$.
\begin{thm}\label{closure}
If $d\ge 2$,
\begin{equation}
\Sigma_{\mathcal S}(\mathcal C)= \overline{\Phi_{F,S}({\mathcal C})},
\end{equation}
where the closure is in the sense of Zariski.

Moreover, $\Phi_{F,S}$
maps generic $m\in\mathcal C$ to the corresponding
point of $\Sigma_{\mathcal S}(\mathcal C)$. 
\end{thm}

Let ${\mathcal I}:=[1:i:0]$ and ${\mathcal J}:=[1:-i:0]$ be the two cyclic points 
of $\mathbb{P}^{2}$. We also define $I:=(1,i,0)$ and $J:=(1,-i,0)$. Points
$\mathcal I$ and $\mathcal J$ will play a 
particular role in our study (see theorems below). 
They will be crucial in the construction of the reflected lines
$\mathcal R_m$. Moreover, 
we will see that 
\begin{equation}\label{casIJ}
\Sigma_{\mathcal I}(\mathcal C)=\{\mathcal J\} \ \mbox{and}\ \ 
 \Sigma_{\mathcal J}(\mathcal C)=\{\mathcal I\}.
\end{equation}
We will use Theorem \ref{closure} and a general fundamental lemma to express the degree of 
$\Sigma_{\mathcal S}(\mathcal{C})$ in terms of some intersection numbers 
(computed in Proposition \ref{calculs}).
Before giving our formulas, let us introduce some notations.

We write $Sing(\mathcal C)$ the set of singular points of $\mathcal C$
(i.e. the set of points $m=[x:y:z]\in\mathcal C$ such that $DF(x,y,z)=0$)
and $Reg(\mathcal C):=\mathcal C\setminus Sing(\mathcal C)$.
We denote by $\mathcal T_{m_1}\mathcal C$ the tangent line to $\mathcal C$ at $m_1$
when $m_1$ is non-singular. We also write $i_{m_1}(\cdot,\cdot)$ the intersection numbers.
For
every $m_1\in\mathcal C$, we write 
$Branch_{m_1}(\mathcal C)$ the set of branches of $\mathcal C$ at $m_1$.
Now, for every $m_1\in \mathcal C$
and every $\mathcal B\in Branch_{m_1}(\mathcal C)$, $\mathcal T_{m_1}\mathcal B$ denotes
the tangent line to $\mathcal B$ at $m_1$ and $e_{(\mathcal B)}$ the multiplicity of $\mathcal B$.
If $m_1\in Sing(\mathcal C)$, such a line $\mathcal T_{m_1}\mathcal B$ will be called
a singular tangent line to $\mathcal C$ at $m_1$. 
The quantity ${\mathbf 1}_{\mathcal P}$ is equal to 1 if property $\mathcal P$ is
true and 0 otherwise.

We will denote by $\textrm{mdeg}(\Sigma_\S(\C))$ the degree with multiplicity of $\Sigma_\S(\C)$.
We have
$${\textrm{mdeg}} (\Sigma_{\mathcal S}(\mathcal C))=\delta_1(\S,\C)\times
{\textrm{deg}} (\Sigma_{\mathcal S}(\mathcal C))$$
(with convention $\infty\times 0=0$),
where ${\textrm{deg}} (\Sigma_{\mathcal S}(\mathcal C))$ is the degree of 
the algebraic curve $\Sigma_\S(\C)$
and where $\delta_1(\S,\C)$ is the degree of $\Phi_{F,S}$.
We recall that $\delta_1(\S,\C)$ corresponds to the number of preimages
on $\mathcal C$ of a generic point of $\Sigma_\S(\C)$ by $\Phi_{F,S}$.
The fact that ${\textrm{mdeg}} (\Sigma_{\mathcal S}(\mathcal C))=0$ means that
$\Sigma_\S(\C)$ contains a single point.

We start with the generic and simple case when 
$\mathcal S$, $\mathcal I$ and $\mathcal J$ do not belong to a singular tangent
line to $\mathcal C$. 
\begin{thm}\label{degrecaustique1}
Assume that $d\ge 2$ and that $\mathcal S$ is equal neither to $\mathcal I$ nor to
$\mathcal J$.
Assume moreover that $\mathcal S$ is not contained in a singular tangent
line to $\mathcal C$ and that $\ell_\infty$ is not a singular tangent line to $\mathcal C$.
Then
$$\mathop{\textrm{mdeg}} (\Sigma_{\mathcal S}(\mathcal C))=3d^\vee-v_1-v_2-v'_2-v_3{\mathbf 1}
      _{\mathcal S\in\mathcal C}-v_4{\mathbf 1}_{\mathcal S\not\in\ell_\infty}, $$
where $d^\vee$ is the class of $\mathcal C$ (i.e. the degree of its dual curve) and with
$$v_1:=\sum_{m_1\in Sing(\mathcal C)\setminus((\mathcal I\mathcal S)\cup(\mathcal J\mathcal S))}
  \sum_{\mathcal B\in Branch_{m_1}(\mathcal C)}\min(i_{m_1}
     (\mathcal B,\mathcal T_{m_1}\mathcal B)-2e_{(\mathcal B)},0), $$
$$v_2:=\sum_{m_1\in Reg(\mathcal C)\setminus\{\mathcal S\}:\mathcal S\in\mathcal T_{m_1}
    \mathcal C,i_{m_1}(\mathcal C,\mathcal 
T_{m_1}\mathcal C)\ne 2} [i_{m_1}(\mathcal C,\mathcal 
T_{m_1}\mathcal C)(1+{\mathbf 1}_{\mathcal I\in \mathcal 
T_{m_1}\mathcal C}+{\mathbf 1}_{\mathcal J\in \mathcal 
T_{m_1}\mathcal C})- 2-{\mathbf 1}_{\mathcal I,\mathcal J\in \mathcal 
T_{m_1}\mathcal C}] ;$$
$$v'_2:=\sum_{m_1\in Reg(\mathcal C)\setminus\{\mathcal S\}:\mathcal S\in\mathcal T_{m_1}
    \mathcal C,i_{m_1}(\mathcal C,\mathcal 
T_{m_1}\mathcal C)= 2} 3\times {\mathbf 1}_{\{\mathcal I,\mathcal J\}\cap \mathcal 
T_{m_1}\mathcal C\ne\emptyset} ;$$
$$v_3:=\left\{\begin{array}{cc}
i_{\mathcal S}(\mathcal C,\mathcal 
T_{\mathcal S}\mathcal C)+ (i_{\mathcal S}(\mathcal C,\mathcal 
T_{\mathcal S}\mathcal C)-1)({\mathbf 1}_{\mathcal I\in \mathcal 
T_{\mathcal S}\mathcal C}+{\mathbf 1}_{\mathcal J\in \mathcal 
T_{\mathcal S}\mathcal C})& \mbox{if}\ 
i_{\mathcal S}(\mathcal C,\mathcal T_{\mathcal S}\mathcal C)\ne 2\\
2+{\mathbf 1}_{\mathcal I\in \mathcal 
T_{\mathcal S}\mathcal C}+{\mathbf 1}_{\mathcal J\in \mathcal 
T_{\mathcal S}\mathcal C}+2\times{\mathbf 1}_{\mathcal I,\mathcal J\in \mathcal 
T_{\mathcal S}\mathcal C}&\mbox{if}\ 
i_{\mathcal S}(\mathcal C,\mathcal T_{\mathcal S}\mathcal C)= 2 ;
\end{array}\right.
$$
$$v_4:=\sum_{m_1\in Reg(\mathcal C):\mathcal T_{m_1}\mathcal C=\ell_\infty}
[i_{m_1}(\mathcal C,\mathcal 
T_{m_1}\mathcal C)-2+\mathbf 1_{m_1\in\{\mathcal I,\mathcal J\}}].$$
\end{thm}
Now let us give the more general but also more technical result. 
In this result we do not distinguish singular and
non-singular points of $\mathcal C$. We recall that, at a non-singular point
$m_1$, $\mathcal C$ admits a single branch and that the multiplicity of this branch
is equal to 1. Let us consider the set $\mathcal E$ 
of all possible
couples $(m_1,\mathcal B)$ with $m_1\in\mathcal C$ and $\mathcal B$ a branch of 
$\mathcal C$ at $m_1$.
Given $(m_1,\mathcal B)\in\mathcal E$, if $i_{m_1}(\mathcal B,\mathcal 
T_{m_1}\mathcal B)=2e_{(\mathcal B)}$, 
we define
$$\gamma_1(m_1,\mathcal B):=\min(3 e_{(\mathcal B)},i_{m_1}(\mathcal B,\mathcal C')),$$ 
where $\mathcal C'$ is any curve non-singular at $m_1$ such that 
$i_{m_1}(\mathcal B,\mathcal C')>2 e_{(\mathcal B)}$
(for example, one can take for $\mathcal C'$ the osculating "circle" $\mathcal O_{m_1}(\mathcal B)$ 
of any pro-branch
of $\mathcal B$, see section \ref{rappels} for the definition of pro-branches) and
$$\gamma_2(m_1,\mathcal B):=\min(\gamma_1(m_1,\mathcal B)-2e_{(\mathcal B)},2e_{(\mathcal B)})
\ \ \mbox{if}\ \ \gamma_1(m_1,\mathcal B)\ne 3e_{(\mathcal B)}$$
and
$$\gamma_2(m_1,\mathcal B):=\min(i_{m_1}(\mathcal B,\mathcal C'' )-2e_{(\mathcal B)},
     2e_{(\mathcal B)})
\ \ \mbox{if}\ \ \gamma_1(m_1,\mathcal B)= 3e_{(\mathcal B)},$$
where $\mathcal C''$ is any algebraic curve non-singular at $m_1$
such that $i_{m_1}(\mathcal B,\mathcal C'')>3 e_{(\mathcal B)}$.

\begin{thm}\label{degrecaustique2}
Assume that $d\ge 2$ and that $\mathcal S$ is equal neither to $\mathcal I$ nor to
$\mathcal J$.
Then
$$\mathop{\textrm{mdeg}} (\Sigma_{\mathcal S}(\mathcal C))=3d^\vee-v_1-v_2-v'_2-(v_3+v'_3){\mathbf 1}
      _{\mathcal S\in\mathcal C}-v_4{\mathbf 1}_{\mathcal S\not\in\ell_\infty}, $$
where $d^\vee$ is the class of $\mathcal C$ (i.e. the degree of its dual curve) and with
$$v_1:=\sum_{(m_1,\mathcal B)\in\mathcal E:m_1\in Sing(\mathcal C)\setminus[ (\mathcal I\mathcal S)
      \cup(\mathcal J\mathcal S)],S\not\in \mathcal T_{m_1}\mathcal B,
     \mathcal T_{m_1}\mathcal B\ne\ell_\infty
    }\!\!\!\!\!\!\!\!\!\!\!\min(i_{m_1}(\mathcal B,\mathcal T_{m_1}\mathcal B)-2e_{(\mathcal B)},0), $$
$$v_2:=\sum_{(m_1,\mathcal B)\in \mathcal E:m_1\ne\mathcal S,\mathcal S\in
     \mathcal T_{m_1}\mathcal B,i_{m_1}(\mathcal B,\mathcal 
T_{m_1}\mathcal B)\ne 2e_{(\mathcal B)}}
\!\!\!\!\!\!\!\!\!\!\!\!\!\!\!\!\!\!\!\!\!\!\!\!\!\!\!\!\!\!\!\!\!
   [i_{m_1}(\mathcal B,\mathcal 
T_{m_1}\mathcal B)(1+{\mathbf 1}_{\mathcal I\in \mathcal 
T_{m_1}\mathcal B}+{\mathbf 1}_{\mathcal J\in \mathcal 
T_{m_1}\mathcal B})- (2+{\mathbf 1}_{\mathcal I,\mathcal J\in \mathcal 
T_{m_1}\mathcal B})e_{(\mathcal B)}] ;$$
$$v'_2:=\sum_{(m_1,\mathcal B)\in \mathcal E:m_1\ne\mathcal S,\mathcal S\in
     \mathcal T_{m_1}\mathcal B,i_{m_1}(\mathcal B,\mathcal 
T_{m_1}\mathcal B)= 2e_{(\mathcal B)}} 
\!\!\!\!\!\!\!\!\!\!\!\!\!\!\!\!\!\!\!\!\!\!\!\!\!\!\!\!\!\!\!\!\!
[\gamma_1(m_1,\mathcal B) {\mathbf 1}_{\#(\{\mathcal I,\mathcal J\}\cap \mathcal 
T_{m_1}\mathcal B)=1}+3 e_{(\mathcal B)}{\mathbf 1}_{\mathcal T_{m_1}\mathcal B=\ell_\infty}] ;$$
$$v_3:=\sum_{(\mathcal S,\mathcal B)\in\mathcal E: 
i_{\mathcal S}(\mathcal B,\mathcal T_{\mathcal S}\mathcal B)\ne 2 e_{(\mathcal B)}}
\!\!\!\!\!\!\!\!\!\!\!\!\!\!
\left[i_{\mathcal S}(\mathcal B,\mathcal 
 T_{\mathcal S}\mathcal B)+ (i_{\mathcal S}(\mathcal B,\mathcal 
 T_{\mathcal S}\mathcal B)-e_{(\mathcal B)})({\mathbf 1}_{\mathcal I\in \mathcal 
 T_{\mathcal S}\mathcal B}+{\mathbf 1}_{\mathcal J\in \mathcal 
 T_{\mathcal S}\mathcal B})\right];$$
$$v'_3:=\sum_{(\mathcal S,\mathcal B)\in\mathcal E: 
i_{\mathcal S}(\mathcal B,\mathcal T_{\mathcal S}\mathcal B)= 2 e_{(\mathcal B)}}
\!\!\!\!\!\!\!\!\!\!\!\!\!\!
\left[(2+{\mathbf 1}_{\mathcal I\in \mathcal 
T_{\mathcal S}\mathcal B}+{\mathbf 1}_{\mathcal J\in \mathcal 
T_{\mathcal S}\mathcal B})e_{(\mathcal B)}+
\gamma_2(m_1,\mathcal B)
{\mathbf 1}_{\mathcal I,\mathcal J\in \mathcal 
T_{\mathcal S}\mathcal B}\right];
$$
$$v_4:=\sum_{(m_1,\mathcal B)\in\mathcal E:\mathcal T_{m_1}\mathcal B=\ell_\infty}
[i_{m_1}(\mathcal B,\mathcal 
T_{m_1}\mathcal B)+(\mathbf 1_{m_1\in\{\mathcal I,\mathcal J\}}-2)e_{(\mathcal B)}].$$
\end{thm}
\begin{rqe}\label{REMARQUE}
Denote by $Flex(\mathcal C)$ the set of inflection points of $\mathcal C$, i.e. the 
set of non-singular points of $\mathcal C$ such that $i_{m_1}(\mathcal C,\mathcal T_{m_1}
   (\mathcal C))>2$.
Recall that, since $\mathcal C$ is irreducible and if $d\ge 2$, 
(by the Bezout theorem and Corollary \ref{Plucker} below) we have
$$3d(d-2)=\sum_{m_1\in\mathcal C}i_{m_1}(\mathcal C,V(H_F))$$
$$=3[d(d-1)-d^\vee]
    +\sum_{m_1\in Sing(\mathcal C)}\sum_{\mathcal B\in Branch_{m_1}(\mathcal C)}
       [i_{m_1}(\mathcal B,\mathcal T_{m_1}\mathcal B)-2 e_{(\mathcal B)}]
    +\sum_{m_1\in Flex(\mathcal C)} (i_{m_1}(\mathcal C,\mathcal T_{m_1}\mathcal C)-2)$$
and so
$$3d^\vee-\sum_{(m_1,\mathcal B)\in\mathcal E,m_1\in Sing(\mathcal C)}
     (i_{m_1}(\mathcal B,\mathcal T_{m_1}\mathcal B)-2 e_{(\mathcal B)})
     =3d+\sum_{m_1\in Flex(\mathcal C)}(i_{m_1}(\mathcal C,\mathcal T_{m_1}\mathcal C)-2).$$
Hence, under assumptions of Theorem \ref{degrecaustique1}, if we suppose moreover that 
$Sing(\mathcal C)\cap((\mathcal I\mathcal S)\cup(\mathcal J\mathcal S))
=\emptyset$ and that, for every 
$m_1\in Sing(\mathcal C)$ and every branch $\mathcal B$ of $\mathcal C$ at $m_1$, we have
     $i_{m_1}(\mathcal B,\mathcal T_{m_1}\mathcal B)\le 2 e_{(\mathcal B)}$ 
(this is true for instance, if all the singular points of $\mathcal C$ are ordinary
cusps and nodes), we get that
$$3d^\vee-v_1 
   =3d+\sum_{m_1\in Flex(\mathcal C)}(i_{m_1}(\mathcal C,\mathcal T_{m_1}\mathcal C)-2).$$
\end{rqe}
\begin{coro}\label{CORO1}
Assume that $d\ge 2$ and that $\mathcal S$ is equal neither to $\mathcal I$ nor to
$\mathcal J$.
Assume that $\mathcal S$ is not contained in a singular tangent
line to $\mathcal C$ and that $\ell_\infty$ is not a singular tangent to $\mathcal C$.

Assume moreover that $Sing(\mathcal C)\cap((\mathcal I\mathcal S)\cup(\mathcal J\mathcal S))
=\emptyset$ and that, for every 
$m_1\in Sing(\mathcal C)$ and every branch $\mathcal B$ of $\mathcal C$ at $m_1$, we have
     $i_{m_1}(\mathcal B,\mathcal T_{m_1}\mathcal B)\le 2 e_{(\mathcal B)}$.

If $\mathcal S\not\in\ell_\infty$, then
$$\mathop{\textrm{mdeg}} (\Sigma_{\mathcal S}(\mathcal C))=3d
      +i_0-t_0-n_0-2\times {\mathbf 1}_{\mathcal S\in\mathcal C}+
     {\mathbf 1}_{\mathcal S\in\mathcal C;\{\mathcal I,\mathcal J\}\cap\mathcal T_{\mathcal S}
        \mathcal C\ne\emptyset}-{\mathbf 1}_{\mathcal I\in\mathcal C,\mathcal T_{\mathcal I}
         \mathcal C=\ell_\infty}-{\mathbf 1}_{\mathcal J\in\mathcal C,\mathcal T_{\mathcal J}
         \mathcal C=\ell_\infty},$$
where
\begin{itemize}
\item $i_0$ is the number of inflection points $m_1$ of $\mathcal C$ such that 
$\mathcal T_{m_1}\mathcal C$ does not contain $\mathcal S$ and is not equal to $\ell_\infty$ :
$$i_0:=\sum_{m_1\in Flex(\mathcal C),\mathcal S\not\in\mathcal T_{m_1}\mathcal C,
\mathcal T_{m_1}\mathcal C\ne \ell_\infty} (i_{m_1}(\mathcal C,\mathcal T_{m_1}\mathcal C)-2)$$
\item $t_0$ is the number of tangencies of $\mathcal C$ with $(\mathcal I\mathcal S)$
or $(\mathcal J \mathcal S)$ :
$$t_0:=\sum_{m_1\in Reg(\mathcal C):
   \mathcal T_{m_1}\mathcal C\subseteq(\mathcal I\mathcal S)\cup
    (\mathcal J\mathcal S)} i_{m_1}(\mathcal C,\mathcal T_{m_1}\mathcal C); $$
\item $n_0$ is the cardinality of the set of non-singular $m_1\in\mathcal C\setminus
  (Flex(\mathcal C)\cup\{\mathcal S\})$ such that $\mathcal T_{m_1}\mathcal C$ is 
$(\mathcal I\mathcal S)$ or $(\mathcal J \mathcal S)$.
\end{itemize}
If $\mathcal S\in\ell_\infty$, then
$$\mathop{\textrm{mdeg}} (\Sigma_{\mathcal S}(\mathcal C))=3d
      +i'_0-t'_0-3\times {\mathbf 1}_{\mathcal S\in\mathcal C}+
       2\times {\mathbf 1}
          _{\mathcal S\in\mathcal C,\mathcal T_{\mathcal S}\mathcal C=\ell_\infty,
             i_{\mathcal S}\ne 2},$$
where
\begin{itemize}
\item $i'_0$ is the number of inflection points $m_1$ of $\mathcal C$ such that 
$\mathcal T_{m_1}\mathcal C$ does not contain $\mathcal S$ :
$$i'_0:=\sum_{m_1\in Flex(\mathcal C),\mathcal S\not\in\mathcal T_{m_1}\mathcal C} 
       (i_{m_1}(\mathcal C,\mathcal T_{m_1}\mathcal C)-2);$$
\item $t'_0$ is given by
$$t'_0:=\sum_{m_1\in Reg(\mathcal C):\mathcal T_{m_1}\mathcal C=\ell_\infty} 
      (2\times i_{m_1}(\mathcal C,\mathcal T_{m_1}\mathcal C)-1).$$
\end{itemize}
\end{coro}
\subsection*{An example}
We give now
an example in order to show how our formula can be used in practice.

We consider the quintic curve $\mathcal C=V(F)$ with $F(x,y,z)=y^2z^3-x^5$.
This curve admits two singular points: $A_1:=[0:0:1]$ and $A_2:=[0:1:0]$, we have $d=5$.

In the chart $z=1$, at $A_1$, $\mathcal C$ has a single branch $\mathcal B_{A_1}$,
which has equation $y^2-x^5=0$ and multiplicity 2.
Hence, $\mathcal C$ admits two pro-branches at $A_1$ of equations $(y=g_i(x),\, i=1,2)$ 
with $g_1(x):=x^{5/2}$ and $g_2(x):=-x^{5/2}$. The tangential intersection number
$i_{A_1}$ of the branch $\mathcal B_{A_1}$ is equal to
$$i_{A_1}=\sum_{j=1}^2\textrm{val}(g_i)=5. $$

In the chart $y=1$, at $A_2$, $\mathcal C$ has a single branch $\mathcal B_{A_2}$,
which has equation $z^3-x^5=0$ and multiplicity 3.
Hence, $\mathcal C$ admits three pro-branches at $A_2$ of equations $(z=h_i(x),\, i=1,2,3)$ 
with $h_1(x):=x^{5/3}$, $h_2(x):=jx^{5/3}$ and  $h_3(x):=j^2x^{5/3}$ (where $j$
is a fixed complex number satisfying
$1+j+j^2=0$). The tangential intersection number
$i_{A_2}$ of the branch $\mathcal B_{A_2}$ is equal to
$$i_{A_2}=\sum_{j=1}^3\textrm{val}(h_i)=5. $$

According to Corollary \ref{Plucker}, the class $d^\vee$ of $\mathcal C$ is given by
$$d^\vee=d(d-1)-\sum_{i,j\in\{1,2\}:i\ne j}val(g_i-g_j)-
         \sum_{i,j\in\{1,2,3\}:i\ne j}val(h_i-h_j)=5\times 4-2\times\frac 52
       -6\times\frac 53=5. $$
Using again Corollary \ref{Plucker}, we know that the number of inflection points
of $\mathcal C$ (computed with multiplicity) is equal to
$$3d(d-2)-3\left[\sum_{i,j\in\{1,2\}:i\ne j}val(g_i-g_j)+
         \sum_{i,j\in\{1,2,3\}:i\ne j}val(h_i-h_j)\right]=0. $$
Therefore, $\mathcal C$ has no inflection points.

The curve $\mathcal C$ admits six isotropic non-singular tangent lines 
(three for $\mathcal I$ and three for $\mathcal J$), which are pairwise distinct.

We consider a light point $\mathcal S=[x_0:y_0:z_0]
   \in\mathbb P^2\setminus\{\mathcal I,\mathcal J\}$.
We will see that
\begin{eqnarray*}
\textrm{mdeg}(\Sigma_{\mathcal S}(\mathcal C))&=& 
   15-{\mathbf 1}_{y_0=0,x_0\ne 0}
    -6\times {\mathbf 1}_{z_0=0,x_0\ne 0} -3\times n_0-5\times{\mathbf 1}_{\mathcal S=A_1}
       -9\times{\mathbf 1}_{\mathcal S=A_2}+{\mathbf 1}_{z_0\ne 0}\\
  &\ &           \ \ \ \ \ \ \ \ \ \ \ \ \ \ \ \ \ \ \ 
        -2\times{\mathbf 1}_{\mathcal S\in Reg(\mathcal C),\mathcal T_{\mathcal S}
          \mathcal C\not\subseteq (\mathcal I\mathcal S)\cup(\mathcal J\mathcal S)},
\end{eqnarray*}
where $n_0$ is the number of non-singular isotropic tangent lines to $\mathcal C$
containing $\mathcal S$.
Hence, we have
$$
\begin{array}{|c|c|}
\hline
\mbox{Condition on }\mathcal S&\textrm{mdeg}(\Sigma_{\mathcal S}(\mathcal C))=\\
\hline
\mbox{for generic }\mathcal S\in{\mathbb P}^2&16\\
\hline
\mbox{for generic }\mathcal S\in{\mathcal C}&14\\
\hline
\mbox{for generic }\mathcal S\in \mathcal T_{A_1}\mathcal B_{A_1}&15\\
\hline
\mbox{for generic }\mathcal S\in \mathcal T_{A_2}\mathcal B_{A_2}=\ell_\infty&9\\
\hline
\mathcal S\in \mathcal T_{A_1}\mathcal B_{A_1}\cap \mathcal T_{A_2}\mathcal B_{A_2}&8\\
\hline
\mathcal S=A_1\mbox{ (double point)}&11\\
\hline
\mathcal S=A_2\mbox{ (triple point)}&6\\
\hline
\mbox{for generic }\mathcal S\mbox{ on a single isotropic tangent}&13\\ 
\hline
\mathcal S\mbox{ on two isotropic tangents}&10\\ 
\hline
\mathcal S=\mathcal I\mbox{ or }\mathcal S=\mathcal J&0\\
\hline
\end{array}
$$
Let us prove the above formula. 
According to Theorem \ref{degrecaustique2}, we have
$$\textrm{mdeg}(\Sigma_{\mathcal S}(\mathcal C))= 3d^\vee-v_1-v_2-v'_2-(v_3+v'_3){\mathbf 1}
   _{\mathcal S\in\mathcal C}-v_4{\mathbf 1}_{z_0\ne 0}.$$
\begin{itemize}
\item $d^\vee=5$ (see above).
\item Since $A_1$ is the single singular point of $\mathcal C$ outside $\ell_\infty$, we have
$$v_1= \min(i_{A_1}-2\times 2,0)
    {\mathbf 1}_{\mathcal S\not\in(A_1\mathcal I)\cup(A_1\mathcal J), 
y_0\ne 0}=0.$$
\item The couples $(m_1,\mathcal B)\in\mathcal E$ that may contribute to $v_2$ corresponds to
inflection points or to singular tangent.
Since $\mathcal C$ admits no inflection points, since the singular tangent line at $A_1$
contains neither $\mathcal I$ nor $\mathcal J$ 
and since the singular tangent line at $A_2$ contains $\mathcal I$
and $\mathcal J$, we get
$$v_2=(i_{A_1}-2\times 2){\mathbf 1}_{y_0=0,x_0\ne 0}+(3i_{A_2}-3\times 3)
      {\mathbf 1}_{z_0=0,x_0\ne 0}$$
\item Since $\mathcal C$ has no point on $\ell_\infty$
except $A_2$, we have
$$v'_2=3\#\{m_1\in Reg(\mathcal C)\setminus\{\mathcal S\}: 
   \mathcal S\in(m_1\mathcal I)\cup(m_1\mathcal J)\}.$$
\item Since $\mathcal C$ admits no inflection points, the only points that possibly contributes
to $v_3$ are the singular points. We have
$$v_3=i_{A_1}\times{\mathbf 1}_{\mathcal S=A_1}+
    (3i_{A_2}-2\times 3){\mathbf 1}_{\mathcal S=A_2}$$
\item Since the non singular point of $\mathcal C$ with isotropic tangent are not in $\ell_\infty$,
we have 
$$v'_3=3\times {\mathbf 1}_{\mathcal S\in Reg(\mathcal C),
       \mathcal T_{\mathcal S}\mathcal C\subseteq(\mathcal I\mathcal S)
      \cup(\mathcal J\mathcal S)}+2\times {\mathbf 1}_{\mathcal S\in Reg(\mathcal C),
     \mathcal T_{\mathcal S}\mathcal C\not\subseteq(\mathcal I\mathcal S)
      \cup(\mathcal J\mathcal S)}.$$
\item $v_4=i_{A_2}-2\times 3$.
\end{itemize}
\section{Caustic by reflection and reflected lines}\label{defcaustique}
Let us consider a light position $\mathcal S=[x_{0}:y_{0}:z_{0}]\in {\mathbb P}^{2}$
and an irreducible algebraic (mirror) curve $\mathcal{C}=V(F)$ of $\mathbb{P}^{2}$
given by  a homogeneous polynomial $F$  of degree $d\ge 2$. 

\begin{defi}
The {\bf caustic by reflection} $\Sigma _{\mathcal S}(\mathcal{C})$ is the 
Zariski closure of the envelope of the
reflected lines $\{{\mathcal R}_{m};m\in Reg(\mathcal{C})\setminus
(\{\mathcal S\}\cup\ell_\infty)\}$, where $\mathcal R_m$ is
the reflected line at $m$ of an incident line coming from
$\mathcal S$ after reflection on $\mathcal C$.
\end{defi}
Let us define the reflected lines $\mathcal R_m$.
Since our problem is euclidean, we endow $\mathbb{P}^{2}$ with an angular
structure for which ${\mathcal I}=[1:i:0]$ and ${\mathcal J}=[1:-i:0]$ 
play a particular role. To this end, let
us recall the definition of the cross-ratio $\beta $ of 4 points of $\ell_\infty$. 
Given four points
$(P_i=[a_i:b_i:0])_{i=1,...,4}$ such that each point appears at most 2 times,
we define the cross-ratio $\beta(P_1,P_2,P_3,P_4)$ of these four points as follows~:
\begin{equation*}
\beta (P_{1},P_{2},P_{3},P_{4})=\frac{(b_3a_1-b_1a_3)(b_4a_2-b_2a_4)}{
    (b_3a_2-b_2a_3)(b_4a_1-b_1a_4)},
\end{equation*}
with convention $\frac{1}{0}=\infty $.

For any distinct lines $A$ and $B$ not equal to $\ell_\infty$, containing
neither $\mathcal I$ nor $\mathcal J$, 
we define the oriented angular measure between $A$
and $B$ by $\theta \in [0;\pi[$
such that $e^{2i\theta }=\beta(a,b,{\mathcal I},{\mathcal J})$  (where $a$ is the 
point of $A$ at infinity
and where $b$ is the point of $B$ at infinity).

For every $m=[x:y:z]\in Reg({\mathcal C})\setminus \{{\mathcal S}\}$ with $z\ne 0$, 
we define the reflected line ${\mathcal R}_m$ at $m$ as follows.
Let ${\mathcal T}_{m}\mathcal{C}$ be the tangent line to $\mathcal{C}$\ at $m$
(with equation $F_{x}X+F_{y}Y+F_{z}Z=0$ and  with point $t_{m}=[F_{y}:-F_{x}:0]$ at infinity).
The incident line at $m$ is line $({\mathcal S}m)$. We denote by $s_m:=[\bar x:\bar y:0]$
its point at infinity.
We have $s_m:=[x_0z-z_0x:y_0z-z_0y:0]$ if $\mathcal S\not\in\ell_\infty$ and
$s_m=[x_0:y_0:0]$ if $\mathcal S\in\ell_\infty$. 
When $s_m$ and $t_m$ are equal neither to $\mathcal I$ nor $\mathcal J$, we define
the reflected line ${\mathcal R}_{m}$ at $m\in \mathcal{C}$ as the line $(mr_m)$
with point $r_m$ at infinity given by the
Snell-Descartes reflection law $Angle(({\mathcal S}m),{\mathcal T}_{m})
 = Angle({\mathcal T}_{m},{\mathcal R}_{m})$,
i.e. 
\begin{equation}\label{birapport}
\beta(s_m ,t_m,{\mathcal I},{\mathcal J})=\beta(t_m,r_{m},{\mathcal I},{\mathcal J}).
\end{equation}
Observe that $r_m$ is well defined by this formula 
as soon as $t_m\not\in\{\mathcal I,\mathcal J\}$. 
In particular, if $({\mathcal S}m)={\mathcal T}_m\mathcal C$, 
then the reflected line at $m$ is $({\mathcal S}m)$.
Moreover, with 
definition (\ref{birapport}), we have
$(s_m=\mathcal I\ \Rightarrow\ r_m=\mathcal J)$ and 
$(s_m=\mathcal J\ \Rightarrow\ r_m=\mathcal I)$.

According to (\ref{birapport}), we have
$$r_m:=[\bar x(F_x^2-F_y^2)+2\bar yF_xF_y:-\bar y(F_x^2-F_y^2)+2\bar x F_xF_y:0]. $$
Hence ${\mathcal R}_m$ is the set of $[X:Y:Z]\in\mathbb P^2$ such that
$$(F_x^2-F_y^2)(-z\bar yX-z\bar xY+Z(\bar yx+y\bar x))+2F_xF_y(z\bar xX-z\bar yY+Z
   (-\bar xx+\bar yy))=0. $$
Let us define $\langle(X_1,Y_1,Z_1),(X_2,Y_2,Z_2)\rangle:=X_1X_2+Y_1Y_2+Z_1Z_2$.
\begin{defi}
Let $m=[x:y:z]\in Reg(\mathcal C)\setminus(\{\mathcal S\}\cup\ell_\infty) $, 
an equation of the reflected line ${\mathcal R}_m$ at $m$ 
is given by $\langle\tilde\rho(x,y,z),(X,Y,Z)\rangle=0$
with
$$\tilde \rho(x,y,z) = \left(\begin{array}{c}
z(z_0y-zy_0)(F_x^2-F_y^2)+2z(zx_0-z_0x)F_xF_y\\
z(z_0x-zx_0)(F_x^2-F_y^2)+2z(z_0y-zy_0)F_xF_y\\
(xzy_0+yzx_0-2z_0xy)(F_x^2-F_y^2)+2(yzy_0-xzx_0+z_0x^2-z_0y^2)F_xF_y\\
\end{array}\right)  \mbox{ if }z_0\ne 0$$
or with
$$\tilde\rho(x,y,z) = \left(\begin{array}{c}
(-zy_0)(F_x^2-F_y^2)+2(zx_0)F_xF_y\\
(-zx_0)(F_x^2-F_y^2)+2(-zy_0)F_xF_y\\
(xy_0+yx_0)(F_x^2-F_y^2)+2(yy_0-xx_0)F_xF_y\\
\end{array}\right)  \mbox{ if }z_0= 0.$$
\end{defi}
\section{Rational map $\Phi_{F,S}$}\label{pptesphi}
\subsection{Proof of Theorem \ref{closure}}
%
\begin{thm}\label{thmcaustique}
Let $\mathcal C=V(F)$ be an algebraic curve of ${\mathbb P}^2$
given by some irreducible homogeneous polynomial $F$ of degree $d\ge 2$ and let 
$S=(x_0,y_0,z_0)\in {\mathbb C}^3\setminus\{0\}$.

For every $m=[x:y:1]\in Reg({\mathcal C})\setminus V(F_x^2+F_y^2)$
which is not a Base point of $\Phi_{F,S}$,
the reflected line ${\mathcal R}_m$ is well defined and is tangent to ${\mathcal C}':=
\overline{\Phi_{F,S}({\mathcal C})}$ at $\Phi_{F,S}(m)$.

Moreover the set of base points of $(\Phi_{F,S})_{\vert\mathcal C}$ is finite.

${\mathcal C}':=\overline{\Phi_{F,S}({\mathcal C})}$
is the caustic by reflection $\Sigma_{\mathcal S}(\mathcal C)$ of $\mathcal C$ with source
point $\mathcal S=[x_0:y_0:z_0]$.
\end{thm}
Before proving this theorem, we explain how the expression $\tilde\Phi(x,y,1)$ can 
be simplified when $(x,y,1)$ is in $\mathcal C$.
Let us recall that, since $F$ is homogeneous with degree $d$, we have
$$d\, F=xF_x+yF_y+zF_z $$
and therefore
$$(d-1)F_z=xF_{xz}+yF_{yz}+zF_{zz},\ (d-1)F_x=xF_{xx}+yF_{xy}+zF_{xz}\ \mbox{and}\
    (d-1)F_y=xF_{xy}+yF_{yy}+zF_{yz}.$$
\begin{rqe}
Thanks to the expression of $F_{zz}$, $F_{xz}$, $F_{yz}$, $F_z$, we have
$$z^2H_F=(d-1)^2 h_F\ \ \ \mbox{on}\ \mathcal C,$$
with
$$h_F:=2F_{xy}F_yF_x-F_{xx}F_y^2-F_{yy}F_x^2
  \ \ \mbox{and}\ \
   z\Delta_SF(x,y,1)=(zx_0-x)F_x+(zy_0-y)F_y.$$
Therefore, for any $m=[x:y:1]\in{\mathcal C}$, we have
$$\tilde\Phi_{F,S}(x,y,1)=\left(\begin{array}{c}-2xh_FN_S+(F_y^2x_0-F_x^2(x_0-2xz_0)
-2F_xF_y(y_0-yz_0)\Delta_SF\\
-2yh_FN_S+(F_x^2y_0z-F_y^2(y_0-2yz_0)-2F_xF_y(x_0-xz_0))\Delta_SF\\
-2h_FN_S+z_0(F_x^2+F_y^2)\Delta_SF\end{array}\right) .$$ 
\end{rqe}

\begin{proof}[Proof of Theorem \ref{thmcaustique}]
Let us first observe that, since $F$ is irreducible of degree $d\ge 2$,
we have ${\mathcal C}\not\subseteq\{F_x^2+F_y^2=0\}\cup\ell_\infty$.

Let us write $\tilde\rho_1$, $\tilde\rho_2$ and $\tilde\rho_3$ the coordinates of $\tilde\rho$.
Observe that $z\tilde\rho_3(m)={-x\tilde\rho_1(m)-y\tilde\rho_2(m)}$.

We will prove that, on $z=1$, we have
$(\tilde \rho\wedge W)=(F_x^2+F_y^2)\tilde\Phi$, for some $W$.

Consider now $\tilde m:=(x,y,1)\in{\mathbb C}^3$ be such that
$m=[x:y:1]$ is in $Reg(\mathcal C)\setminus(\{\mathcal S\}\cup\ell_\infty\cup V(F_x^2+F_y^2))$
and is not a base point of $\Phi_{F,S}$. Hence $\tilde\rho(x,y,1)\ne 0$ and 
the reflected line $\mathcal R_m$ is well defined.

To simplify notations, we ommit indices $F,S$ in $\tilde\Phi$ and in $\Phi$.

To prove that $\Phi(m)$ belongs to ${\mathcal R}_m$, it is enough to prove that 
\begin{equation}\label{eq1}\langle\tilde\Phi(\tilde m),\rho(\tilde m)\rangle=0,\end{equation}

If this is true, to prove that ${\mathcal R}_m$ is tangent to ${\mathcal C}'$ at $\Phi(m)$,
it is enough to prove that 
\begin{equation}\label{eq2}
\langle \tilde\Phi(\tilde m),W(m)\rangle=0, \ \mbox{ with }
W:=\left(\begin{array}{c}W_1:=(\tilde\rho_1)_x(-F_y)+(\tilde\rho_1)_yF_x\\
       W_2:=(\tilde\rho_2)_x(-F_y)+(\tilde\rho_2)_yF_x\\W_3:=-xW_1-yW_2-\tilde\rho_1(-F_y)-\tilde\rho_2(F_x)
   \end{array}\right) .\end{equation}
Indeed, let us consider a parametrization $M(t)=[x(t):y(t):1]$ of $\mathcal C$ 
in a neighbourhood of $m$ such that $M(0)=m$ and such that $x'(t)=-F_y$ and
$y'(t)=F_x$. This is possible since $DF(m)$ is non-null.
Then $\varphi(t):=\Phi(M(t))$ is a parametrization of ${\mathcal C}'$ 
at a neighbourhood of $\Phi(m)$.
Let $r(t):=\rho(M(t))$. 
The fact that ${\mathcal R}_m$ is tangent to ${\mathcal C}'$ at $\Phi(m)$ means that 
$\langle r(0),\varphi'(0)\rangle=0$, i.e. $\langle r'(0),\Phi(m)\rangle=0$
(since $\langle r(t),\varphi(t)\rangle=0$, we have $\langle r'(0),\Phi(m)\rangle=
-\langle \rho(m),\varphi'(0)\rangle$). We have $r'(0)=W(\tilde m)$.

Hence, to prove the theorem, it is enough to prove that
$$\boxed{\tilde \rho(\tilde m)\wedge W(\tilde m)=(F_x^2+F_y^2)\tilde\Phi(\tilde m)},$$
which is the key point of this proof.
This can be checked by a fastidious formal computation 
thanks to the following formulas
(and thanks to a symbolic computation software):
$$(\tilde\rho_1)_x=2(z_0y-zy_0)(F_xF_{xx}-F_yF_{xy})-2z_0F_xF_y+2(zx_0-z_0x)
 (F_xF_{xy}+F_{xx}F_y),$$
$$(\tilde\rho_1)_y:=z_0(F_x^2-F_y^2)+ 2(z_0y-zy_0)(F_xF_{xy}-F_yF_{yy})
+2(zx_0-z_0x)(F_{xy}F_y+F_xF_{yy}),$$
$$(\tilde\rho_2)_x:=z_0(F_x^2-F_y^2)+
   2(z_0x-zx_0)(F_xF_{xx}-F_yF_{xy})+2(z_0y-zy_0)(F_xF_{xy}+F_{xx}F_y),$$
$$  (\tilde\rho_2)_y=2(z_0x-zx_0)(F_xF_{xy}-F_yF_{yy})+2z_0F_xF_y+2(z_0y-zy_0)(F_xF_{yy}+F_{xy}F_y).$$
The fact that the set of base points of $\Phi$ is finite on $\mathcal C$ comes 
from the following proposition. The last point follows.
\end{proof}
We can observe that Theorem \ref{thmcaustique} remains true when $d=1$ and
$\mathcal S,\mathcal I,\mathcal J\not\in\mathcal C$.
\subsection{Base points of $(\Phi_{F,S})_{|\mathcal C}$}\label{basepoints}
\begin{rqe}[Light position at $\mathcal I$ or $\mathcal J$]
We notice that $\tilde\Phi_{F,\alpha I}=-\alpha^2(\Delta_IF)^3J$ and 
$\tilde\Phi_{F,\alpha J}=-\alpha^2(\Delta_JF)^3I$. Hence 
$$\Sigma_{\mathcal I}(\mathcal C)=\{\mathcal J\}\ \ \mbox{and}\ \ 
    \Sigma_{\mathcal J}(\mathcal C)=\{\mathcal I\}.$$
This is not surprising since, in these cases, we always have $r_m=\mathcal J$
and $r_m=\mathcal I$ respectively.

Hence, in the sequel, we will suppose that $\mathcal S\in{\mathbb P}^2\setminus\{\mathcal I,
    \mathcal J\}$.
\end{rqe}
\begin{prop}\label{pointsbasecaustique}
Let us assume that hypotheses of theorem \ref{thmcaustique} hold true and that 
$\mathcal S\not\in\{{\mathcal I},{\mathcal J}\}$.

If $\mathcal S\in\ell_\infty$, then $[x:y:z]\in{\mathcal C}$ is a base point of
$\Phi_{F,S}$ if and only if
$$h_F(x,y,z)=0 \ \ \mbox{and}\ \ \Delta_SF(x,y,z)=0. $$

If $\mathcal S\not\in\ell_\infty$, then $[x:y:z]\in{\mathcal C}$ is a base point of
$\tilde\Phi$ if and only if
$$[H_F(x,y,z)=0\ \mbox{or}\ N_S(x,y,z)=0]\ \ \mbox{and}\ \ [\Delta_SF(x,y,z)=0
 \ \mbox{or}\ (z=0\ \mbox{and}\ F_x=F_y=0)\ )]. $$
\end{prop}
\begin{rqe}
This result insures that
the set of base points of $\Phi_{F,S}$ is finite on $\mathcal C=V(F)$
as soon as $F$ is irreducible and of degree $d\ge 2$.
\end{rqe}
\begin{rqe}[Geometric interpretation]
Let us notice that $N_S(x,y,z)=0$ means that either ${\mathcal I},[x:y:z],\mathcal S$ lie on a same line or
${\mathcal J},[x:y:z],\mathcal S$ are on a same line.

If ${\mathcal S}\in \ell_\infty\setminus
   \{ {\mathcal I},{\mathcal J}\}$, the base points of $\Phi_{F,S}$ on $\mathcal C$ are~:
\begin{itemize}
\item the singular points of $\mathcal C$  (since $DF=0$ implies $H_F=0$),
\item the inflection points $m$ of $\mathcal C$ such that ${\mathcal S}$ is in
${\mathcal T}_m{\mathcal C}$,
\item the points $m$ of $\mathcal C$ lying on $\ell_\infty$ 
such that ${\mathcal T}_m{\mathcal C}$
is $\ell_\infty$,
\item $m=\mathcal S$ (if $\mathcal S$ belongs to $\mathcal C$).
\end{itemize}
If ${\mathcal S}\not\in \ell_\infty$, 
the base points of $\Phi_{F,S}$ on $\mathcal C$ are~:
\begin{itemize}
\item the singular points of $\mathcal C$,
\item the inflection points $m$ of $\mathcal C$ such that ${\mathcal S}$ is in
${\mathcal T}_m{\mathcal C}$,
\item the inflection points $m$ of $\mathcal C$ belonging to infinity 
such that ${\mathcal T}_m{\mathcal C}$
is $\ell_\infty$,
\item the points $m$ belongs of $\mathcal C$, such that $\mathcal S$ belongs to 
${\mathcal T}_m{\mathcal C}$ and $m\in({\mathcal S}{\mathcal I})\cup({\mathcal S}{\mathcal J})$, i.e.
\begin{itemize}
\item $m\in\mathcal C\cap\{\mathcal S\}$,
\item $m\in{\mathcal C}\cap\{{\mathcal I},{\mathcal J}\}$
and $\mathcal S\in T_m{\mathcal C}$,
\item $m\in\mathcal C$ such that $T_m{\mathcal C}$ equals an isotropic line 
$(\mathcal S\mathcal I)$ or $(\mathcal S\mathcal J)$,
\end{itemize}
\item $m\in\mathcal C\cap\{{\mathcal I},{\mathcal J}\}$ such that ${\mathcal T}_m{\mathcal C}$ 
is $\ell_\infty$.
\end{itemize}
\end{rqe}

\begin{proof}[Proof of Proposition \ref{pointsbasecaustique}]
We just prove one implication, the other one being obvious. To simplify notations, we write
$\tilde\Phi$ instead of $\tilde\Phi_{F,S}$.

{\bf Let us suppose that $z_0=0$.\/}

In this case, we have $N_S=z^2(x_0^2+y_0^2)$ and so $H_FN_S=(d-1)^2(x_0^2+y_0^2)h_F$ and
$$\Phi(x,y,z)=\left(\begin{array}{c}-{2x(x_0^2+y_0^2)h_F}+
    (F_y^2x_0-F_x^2x_0-2F_xF_yy_0)\Delta_SF\\
-2y(x_0^2+y_0^2)h_F+(F_x^2y_0-F_y^2y_0-2F_xF_yx_0)\Delta_SF\\
-2z(x_0^2+y_0^2)h_F\end{array}\right).$$
Let $\tilde m=(x,y,z)$ be such that $\tilde\Phi(\tilde m)=0$.
Since $x_0^2+y_0^2\ne 0$, thanks to the last equation, we get that $zh_F=0$, i.e.
$$ h_F=0. $$
According to the two other equations, we get that
$$0=(F_y^2x_0-F_x^2x_0-2F_xF_yy_0)\Delta_SF \ \mbox{and}\ 
(F_x^2y_0-F_y^2y_0-2F_xF_yx_0)\Delta_SF=0.$$
If moreover $\Delta_SF\ne 0$, then we must have $a=0$ and $b=0$ with
$$a:=F_y^2x_0-F_x^2x_0-2F_xF_yy_0\ \mbox{and}\ b:=F_x^2y_0-F_y^2y_0-2F_xF_yx_0. $$
Writing successively $x_0a-y_0b=0$ and $y_0a+x_0b=0$ and using the fact that
$x_0^2+y_0^2\ne 0$, we get that $F_x=F_y=0$, and so $\Delta_SF=0$.

{\bf Let us suppose that $z_0\ne 0$.\/}

Let $\tilde m=(x,y,z)$ be such that $\tilde\Phi(\tilde m)=0$.
We observe that we have
$$\tilde\Phi(\tilde m)=\left(\begin{array}{c}-\frac{2xH_FN_S}{(d-1)^2}+
  ((F_x^2+F_y^2)x_0-2F_x\Delta_SF)\Delta_SF\\
-\frac{2yH_FN_S}{(d-1)^2}+((F_x^2+F_y^2)y_0-2F_y\Delta_SF)\Delta_SF\\
-\frac{2zH_FN_S}{(d-1)^2}+z_0(F_x^2+F_y^2)\Delta_SF\end{array}\right).$$
Hence we have
\begin{equation}\label{decompPhi}
\tilde\Phi(\tilde m)=-\frac{2H_FN_S}{(d-1)^2}\left(\begin{array}{c}x\\y\\z\end{array}\right)+
    (F_x^2+F_y^2)\Delta_SF\left(\begin{array}{c}x_0\\y_0\\z_0\end{array}\right)
      -2(\Delta_SF)^2 \left(\begin{array}{c}F_x\\F_y\\0\end{array}\right)=0.
\end{equation}
First, let us consider the case when $H_FN_S=0$ (which is equivalent to the fact that 
$H_F=0$ or $N_S=0$).
In this case, if $\Delta_SF\ne 0$, then we have $F_x^2+F_y^2=0$
(by the third equation) and therefore $F_x=F_y=0$ (according to the two other equations).
Let us notice that, since $zF_z=dF-xF_x-yF_y$, and since $F_z\ne 0$ (since  
$\Delta_SF\ne 0$), this implies that $z=0$.

Second, let us consider the case when $H_FN_S\ne 0$. Then $\Delta_SF\ne 0$.
Since $xF_x+yF_y+zF_z=F=0$ and according to the definition of $\Delta_SF$, we get
$$0=(F_x\ F_y\ F_z)\cdot\tilde\Phi(\tilde m)= (\Delta_SF)^2(F_x^2+F_y^2)-
  2(\Delta_SF)^2(F_x^2+F_y^2) $$
and so $F_x^2+F_y^2$=0 and so $z=0$ (by $\tilde\Phi_3=0$) and $x^2+y^2=0$
(from $x\tilde\Phi_1+y\tilde\Phi_2=0$) which contradicts the fact that $N_S\ne 0$.
\end{proof}
\subsection{Degree of $\Phi_{F,S}$}
We recall the definition of the degree of a rational map on an irreducible curve.
\begin{defi}
Let $\phi:\mathbb P^p\rightarrow\mathbb P^q$ be a rational map and $\C_1$
an irreducible algebraic curve of $\mathbb P^p$. Let $\mathcal \C_2$ be the Zariski closure
$\overline{\phi(\C_1)}$.

The map $\phi^*:\mathbb C(\C_2)\rightarrow \mathbb C(\C_1)$ defined by $\phi^*(f)=f\circ\phi$
is called the pullback of $\phi_{|\C_1}$.

If $\mathcal C_1\not\subset Base(\phi)$, if
$\phi_{|\C_1\setminus Base(\phi)}$ 
is not constant, the degree of $\phi_{|\C_1}$ is 
the degree $[\mathbb C(\C_1) : \phi^*(\mathbb C(\C_2))]$ of
$\mathbb C(\C_1)$ as a finite algebraic extension of
$\phi^*(\mathbb C(\C_2))$.

If $\mathcal C\not\subset Base(\phi)$ and if
$\phi_{|\mathcal C\setminus Base(\phi)}$ is constant, the degree is equal to infinity.
\end{defi}
The following interpretation of the degree of a rational map is also useful.
\begin{rqe}
Let $\phi:\mathbb P^p\rightarrow\mathbb P^q$ be a rational map and $\C$
an irreducible algebraic curve of $\mathbb P^p$.

We recall that, thanks to blowing up (\cite[Example II-7-17-3]{Hartshorne})
and to a classical morphism result (\cite[Proposition II-2-6]{Silverman},
\cite[Proposition II-6-8]{Hartshorne}), if 
$\mathcal C\not\subset Base(\phi)$, if
$\phi_{|\mathcal C\setminus Base(\phi)}$ 
is not constant and has degree $\delta_1$, then 
there exists a finite set $\mathcal N$ such that
for every point $y$ of $\phi(\mathcal C)\setminus\mathcal N$,
the number of preimages of $y$ by $\phi$ is equal to $\delta_1$. 

If $\delta_1=1$, then 
the map $\phi_{\vert \mathcal C}$ is birational onto its image.

When $\#(\phi(\mathcal C\setminus Base(\phi)))=1$, we set $\delta_1=\infty$.
\end{rqe}

The question of the degree of the caustic map $\Phi_{F,S}$ is not evident,
even if $\S\not\in\ell_\infty$.
Indeed, when $\S\not\in\ell_\infty$, we recall that, as noticed by Quetelet and Dandelin, 
$\Phi_{F,S}$ is the evolute of the $\S$-centered homothety (with ratio 2) of the pedal
of $\C$ from $\S$. 
It is easy to see that the pedal map is birational on any irreducible curve
which is not a line. It is clear that the $\S$-centered homothety (with ratio 2)
is an isomorphism of $\mathbb P^2$.
But, as proved in \cite{Fantechi}, the degree of the evolute map is not 
necessarily equal to 1 or to infinity (contrarily to a statement in \cite{Coolidge}). 
%
%
%
\section{About the computation of the degree of the caustic}\label{sectionlemmefond}
\subsection{A fundamental lemma}
The idea used in this paper to compute the degree of caustics is
is based on the following general lemma giving 
a way to compute the degree of the image of a curve by a rational map.
The proof of the Pl\"ucker formula given in \cite[p. 91]{Fisher} can be seen
as an application of the following lemma.

The following definition extends the notion of polar into a notion of $\varphi$-polar.
\begin{defi}
Let $p\ge 1$, $q\ge 1$. Given $\varphi :{\mathbb P}^{p}\rightarrow {\mathbb P}^q$
be a rational map
given by $\varphi=[\varphi_0:\cdots:\varphi_q]$ 
(with homogeneous polynomial functions
$\varphi_j:{\mathbb C}^{p+1}\rightarrow{\mathbb C}^{q+1}$) and $a=[a_0:\cdots:a_q]\in
{\mathbb P}^q$, we define the $\varphi$-polar at $a$, denoted 
by ${\mathcal P}_{\varphi,a}$, as follows
$${\mathcal P}_{\varphi,a}:=V\left( \sum_{j=0}^q a_j\varphi_j \right).$$
\end{defi}
With this definition, the classical polar of a hypersurface 
${\mathcal C}=V(F)$ of ${\mathbb P}^p$ 
(for some homogeneous polynomial $F$) at $a$ is the $\nabla F$-polar
at $a$, where $\nabla F(X)$ denotes as usual the vector constitued of the partial
derivatives of $F$ at $X\in{\mathbb C}^{p+1}\setminus\{0\}$.

We recall that the set of base points 
of a rational map $\varphi=[\varphi_0:...:\varphi_q]:{\mathbb P}^p\rightarrow{\mathbb P}^q$
is the set
$$Base(\varphi):=\bigcap_{j=0}^qV(\varphi_j).$$
The cardinality of a set $E$ will be written $\#E$.
\begin{lem}[Fundamental lemma]\label{lemmefondamental}
Let $\mathcal C$ be an irreducible algebraic curve of ${\mathbb P}^p$.
Let $p\ge 1$, $q\ge 1$ be two integers and 
$\varphi :{\mathbb P}^p\rightarrow{\mathbb P}^q$ be a rational map 
given by $\varphi=[\varphi_0:\cdots:\varphi_q]$, with 
$\varphi_0,...,\varphi_q\in{\mathbb C}[x_0,...,x_p]$ some
homogeneous polynomials of degree $\delta$.
Assume that $\mathcal C\not\subseteq Base(\varphi)$ and that
$\varphi_{|\mathcal C}$ has degree $\delta_1\in \mathbb N\cup\{\infty\}$.

Then, for generic $a=[a_0:\cdots:a_q]\in{\mathbb P}^q$, the following 
formula holds true
$$\delta_1.\mbox{deg}\left(\overline{\varphi({\mathcal C}})\right)=\delta .\mbox{deg}
({\mathcal C})
     -\sum_{p\in Base(\varphi_{\vert \mathcal C})} i_p \left({\mathcal C},
{\mathcal P}_{\varphi,a}\right),$$
with convention $0.\infty=0$ and 
$\textrm{deg}(\overline{\varphi({\mathcal C}}))=0$ if $\#\overline{\varphi({\mathcal C}})<\infty$.
\end{lem}
Before proving this result, we make some observations.
\begin{lem}
Let $\varphi_0,...,\varphi_{q}\in{\mathbb C}[x_0,...,x_p]$ be homogeneous 
polynomials of degree $\delta$.
Consider $\Phi:=(\varphi_0,...,\varphi_{q})$ and the rational map 
$\varphi :{\mathbb P}^p\rightarrow{\mathbb P}^q$
defined by $\varphi=[\varphi_0:\cdots:\varphi_q]$.
Let $\mathcal K$ be the cone surface associated to an algebraic irreducible
curve $\mathcal C$ of $\mathbb P^{p}$.
If $\#(\varphi(\mathcal C\setminus Base(\varphi)))>1$
and if that $\varphi_{|\mathcal C}$ has degree $\delta_1$, then
the set of regular points $m=[x_0:...:x_p]\in\mathcal C$ such that
    $$D\Phi(\tilde m)(\mathcal T_{\tilde m}\mathcal K)\subset Vect(\Phi(\tilde m)),
     \ \ \mbox{with}\ \ \tilde m:=(x_0,...,x_p)$$
(where $\mathcal T_{\tilde m}\mathcal K$ is the vector 
tangent plane to $\mathcal K$ at $\tilde m$)
is finite. 
\end{lem}
\begin{proof}
Let $\tilde m=(y_0,...,y_p)\in\mathbb C^{p+1}$ be such that
$m:=\Pi(\tilde m)$ is a non-singular point of ${\mathcal C}$.
We know that there exist an integer $s\ge p-1$ and $s$ homogeneous polynomials
$F^{(1)},...,F^{(s)}\in{\mathbb C}[x_0,...,x_p]$ such that 
$\mathcal C=\cap_{i=1}^sV(F^{(i)})$. We also know that $\mathcal T_{\tilde m}\mathcal K$
is $\cap_{i=1}^sV(\sum_{j=0}^pX_jF_{x_j}^{(i)}(\tilde m))$.
Thanks to classical methods of resolution of linear systems, we know that
there exist an integer $r\ge 1$ and $(G_k^{(u)};k=0,...,p;u=1,...,r)$
a family of homogeneous polynomials of $(F_{x_i}^{(j)};i=0,...,p;j=1,...,s)$ such that
$$\mathcal T_{\tilde m}\mathcal K
  =Vect(V_u(\tilde m); \ u=1,...,r)$$
$$ \mbox{with}\ \ 
    V_u(\tilde m):=(G_k^{(u)}(F_{x_i}^{(j)}(\tilde m));i=0,...,p;j=1,...,s)_{k=0,...,p}.$$
Now the fact that 
\begin{equation}\label{degenere}
D\Phi(\tilde m)(\mathcal T_{\tilde m}\mathcal K)\subset Vect(\Phi(\tilde m))
\end{equation}
is equivalent to
$$\forall u=1,...,r,\ \ \forall i,j=0,...,p,\ \ 
     [D\Phi(\tilde m)\cdot V_u(\tilde m)]_i\varphi_j(\tilde m)
     -[D\Phi(\tilde m)\cdot V_u(\tilde m)]_j\varphi_i(\tilde m)=0. $$
and so to the fact that $m$ belongs to some algebraic variety.
Since $\mathcal C$ is irreducible,
we conclude that either the set of such $m$ is finite or is equal to $\mathcal C$.
The fact that this set is $\mathcal C$ would mean that for every
point $\tilde m$ such that $m:=\Pi(\tilde m)$ 
is a non-singular point of $\mathcal C$ and is not a base point of $\varphi$, for 
every $t\mapsto X(t)=(X_0(t),...,X_p(t))$ such that 
$t\mapsto \Pi(X(t))$ is a local parametrization of $\mathcal C$ satisfying $X(0)=m$, we have
$(\Phi\circ X)'(t)=D\Phi(X(t))\cdot X'(t)$ with $X'(t)\in\mathcal T_{X(t)}\mathcal K$ and so
$$(\varphi_i\circ X)'(t)\varphi_j(X(t))-(\varphi_j\circ X)'(t)\varphi_i(X(t))=0.$$
This implies that $\Phi (X(t))=\frac {\varphi_i(X(t))}{\varphi_i(m)}
     \Phi (m)$ if $i$ is such that $\varphi_i(m)\ne 0$.
This means that $\Pi(\Phi(X(t)))=\Pi(\Phi(m))$ for every $t$.
This is impossible by hypothesis.
\end{proof}
\begin{proof}[Proof of Lemma \ref{lemmefondamental}]
Let us consider a polynomial map $\Phi:{\mathbb C}^{p+1}\rightarrow {\mathbb C}^{q+1}$ 
associated to $\varphi$.
For every $a=[a_0:\cdots:a_q]\in{\mathbb P}^q$, we consider the hyperplane 
$V(\psi_a):=\{[y_0:\cdots:y_q]\in{\mathbb P}^q\, :\, \sum_{i=0}^q a_i y_i=0\}$
(with $\psi_a(y_0,\cdots,y_q):=\sum_{i=0}^q a_i y_i$ for some choice of representant 
of $a$). Let  $B:=Base(\varphi_{\vert \mathcal C})$ (i.e. the set of $m=\Pi(\tilde m)\in\mathcal C$
such that $\Phi(\tilde m)=0$). This set clearly belongs to ${\mathcal C}\cap
  V(\psi_a\circ \Phi)$.

\noindent\textbullet\ If $\#(\varphi(\mathcal C\setminus Base(\varphi)))=1$, we consider
the point $M$ of $\mathbb P^q$ such that $\varphi(\mathcal C\setminus Base(\varphi))=\{M\}$
and we observe that, for every $a\in{\mathbb P}^q\setminus V(\psi_M)$, we have
$\mathcal C\cap V(\psi_a\circ\Phi)=B$ and so, by Bezout theorem
(using the fact that $V(\psi_a\circ\Phi)$ 
is an hypersurface that does not contain $\mathcal C$)
\begin{eqnarray*}
\mbox{deg}({\mathcal C})\delta&=&\mbox{deg}({\mathcal C})\mbox{deg}(\psi_a\circ\Phi)\\
 &=&\sum_{m\in {\mathcal C}\cap V(\psi_a\circ\Phi)}i_m({\mathcal C},V(\psi_a\circ\Phi))\\
 &=& \sum_{m\in B}i_m({\mathcal C},V(\psi_a\circ\Phi))\\
 &=&\sum_{m\in B}i_m({\mathcal C},V(\psi_a\circ\Phi))+\delta_1.
\mbox{deg}\left(\overline{\varphi({\mathcal C}})\right),
\end{eqnarray*}
since $\mbox{deg}\left(\overline{\varphi({\mathcal C}})\right)=0$.

\noindent\textbullet\ Assume now that $\#(\varphi(\mathcal C\setminus Base(\varphi)))>1$.

Let $a=[a_0:\cdots:a_q]\in{\mathbb P}^q$ be such that
\begin{enumerate}
\item for every $y\in \overline{\varphi({\mathcal C}\setminus{\mathcal B})} 
    \cap V(\psi_a)$ is in $\varphi({\mathcal C})$ and satisfies 
    $\#\varphi^{-1}(\{y\})=\delta_1$,
\item $\mbox{deg}\left(\sum_{i=0}^q a_i\varphi_i\right)=\delta$,
\item for every $y\in \overline{\varphi({\mathcal C}\setminus{\mathcal B})} 
   \cap V(\psi_a)$, we have
$i_y(\overline{\varphi({\mathcal C}\setminus{\mathcal B})}, V(\psi_a))=1$,
\item $V(\psi_a)$ contains no point $\Pi(\tilde m)$ of $\mathcal C$
such that $D\phi(\tilde m)\cdot(\mathcal T_{\tilde m}\mathcal K)\subseteq Vect(\Phi(\tilde m))$,
where $\mathcal K$ is the cone surface
associated to $\mathcal C$ and where $\mathcal T_{\tilde m}\mathcal K$
is its tangent plane at $m$, 
\item $\#(\overline{\varphi({\mathcal C}\setminus{\mathcal B})} \cap V(\psi_a))<\infty$,
\item $V(\psi_a)\cap \varphi(Sing({\mathcal C}))=\emptyset$, with
$Sing({\mathcal C})=\{m=[x:y:z]\in{\mathcal C}\ :\ DF(x,y,z)= 0\}$.
\end{enumerate}
Now let us write $V(\psi_a\circ\Phi):=\{[y_0:\cdots:y_p]\in{\mathbb P}^p\, 
:\, \sum_{i=0}^q a_i \varphi_i(y_0,...,y_p)=0\}$.

Point 5 insures that $\#({\mathcal C} \cap V(\psi_a\circ\Phi))<\infty$.
Hence, ${\mathcal C}$ and $ V(\psi_a\circ\Phi)$ have no common component and
since $\mathcal C$ is a curve and $V(\psi_a\circ\Phi )$ is an hypersurface, according to Bezout 
theorem, since $V(\psi_a\circ\Phi)$ does not contain $\mathcal C$, we have
\begin{eqnarray*}
\mbox{deg}({\mathcal C})\delta&=&\mbox{deg}({\mathcal C})\mbox{deg}(\psi_a\circ\Phi)\\
 &=&\sum_{m\in {\mathcal C}\cap V(\psi_a\circ\Phi)}i_m({\mathcal C},V(\psi_a\circ\Phi))\\
 &=& \sum_{m\in B}i_m({\mathcal C},V(\psi_a\circ\Phi))+\sum_{m\in
     ({\mathcal C}\setminus B)\cap V(\psi_a\circ\Phi)}i_m({\mathcal C},V(\psi_a\circ\Phi)).
\end{eqnarray*}
Let us now consider any $m=[x_0:\cdots:x_{p}]\in
     ({\mathcal C}\setminus B)\cap V(\psi_a\circ\Phi)$.
We have $\varphi(m)\in\varphi(\mathcal C\setminus B)\cap V(\psi_a)$.
According to point 6, $\mathcal C$ admits a tangent line ${\mathcal T}_m$ at $m$.
Set  $\tilde m:=(x_0,...,x_{p})$.
According to point 4, we consider a tangent vector 
$v=(v_0,\dots,v_{p})\in{\mathbb C}^{p+1}$ to $\mathcal K$
such that $D\Phi(\tilde m)(v)$ and $\Phi(\tilde m)$
are linearly independent. 

Therefore, according to point 3, with the notation $a:=(a_0,...,a_q)$ we have
$$\langle a,(D\Phi)(\tilde m)( v)\rangle\ne 0 \ \mbox{ and so }\
  \langle v,{}^t(D\Phi)(\tilde m)( a)\rangle\ne 0.$$
Since ${}^t(D\Phi)(\tilde m)\cdot a=\nabla(\psi_a\circ\Phi)(\tilde m)$, we get that
$i_m({\mathcal C},V(\psi_a\circ\Phi))=1$.
Hence, we have
\begin{eqnarray*}
\sum_{m\in
     ({\mathcal C}\setminus B)\cap V(\psi_a\circ\Phi)}i_m({\mathcal C},V(\psi_a\circ\Phi))
&=& \#\left(({\mathcal C}\setminus B)\cap V(\psi_a\circ\Phi)\right)\\
&=& \delta_1\#\left(\varphi({\mathcal C}\setminus B))\cap V(\psi_a)\right)\\
&=& \delta_1\#\left(\overline{\varphi({\mathcal C}\setminus B))}\cap V(\psi_a)\right),
    \mbox{ according to Point 1}\\
&=& \delta_1\sum_{y\in
     \overline{\varphi({\mathcal C})}\cap V(\psi_a)}i_y(\overline{\varphi({\mathcal C})},
V(\psi_a)), \mbox{ according to Point 3}\\
&=& \delta_1.\mbox{deg}(\overline{\varphi({\mathcal C})}).
\end{eqnarray*}
\end{proof}
\subsection{About intersection numbers at the base points of $\Phi_{F,S}$}
Theorems \ref{degrecaustique1} and \ref{degrecaustique2} will be a direct consequence
of the fundamental lemma (Lemma \ref{lemmefondamental}) 
and of the computation of $i_{m_1}\left({\mathcal C}, {\mathcal P}_{\Phi,a}\right)$
at every base point.
To compute these intersection numbers, we will compute intersection numbers
of branches thanks to the notion of pro-branches.
It will be important to observe that, since $F_x^2+F_y^2=\Delta_IF.\Delta_JF$,
$\tilde\Phi_{F,S}$ can be rewritten
\begin{equation}\label{expressionPhi}
\tilde\Phi_{F,S}=-\frac{2H_F.f_{I,S}.f_{J,S}}{(d-1)^2}.Id+
    \Delta_IF.\Delta_JF.\Delta_SF. S
      -(\Delta_SF)^2[\Delta_IF.J+\Delta_JF.I],
\end{equation}
where
$f_{A,B}(C)=\textrm{det}(A|B|C)$, for every $A,B,C\in\mathbb C^3$.
We observe that $V(f_{A,B})$ is the line $(\Pi(A)\Pi(B))$.
In the proof of Theorems \ref{degrecaustique1} and \ref{degrecaustique2}, thanks to formula
(\ref{expressionPhi}), we will easily see that, for a generic $a\in{\mathbb P}^2$, 
\begin{equation}\label{MIN} 
i_{m_1}\left({\mathcal C}, {\mathcal P}_{\Phi_{F,S},a}\right)
    \ge \min\left(
   i_{m_1}\left({\mathcal C}, V(\Psi_1)\right),
   i_{m_1}\left({\mathcal C}, V(\Psi_2)\right),
   i_{m_1}\left({\mathcal C}, V(\Psi_3)\right),
   i_{m_1}\left({\mathcal C}, V(\Psi_4)\right)\right),
\end{equation}
with
$$\Psi_1:=  -\frac{2H_F.f_{I,S}.f_{J,S}}{(d-1)^2},\ \ \ \ 
     \Psi_2:=\Delta_IF.\Delta_JF.\Delta_SF,$$
$$\Psi_3:=(\Delta_SF)^2\Delta_IF\ \ \ \mbox{and}\ \ \ \Psi_4:=(\Delta_SF)^2\Delta_JF.$$
More precisely, we will prove the same inequality for branches of $\mathcal C$ at $m_1$
instead of $\mathcal C$, this inequality being an equality in most of the cases 
but not in every case.
This will be detailed in section \ref{preuve}. Before going on in the proof of Theorems
\ref{degrecaustique1} and \ref{degrecaustique2}, let us give some general results on
intersection numbers including general formulas for the following intersection numbers
$$i_{m_1}({\mathcal C},V(H_F))\ \ \ \mbox{and}\ \ \ 
i_{m_1}({\mathcal C},V(\Delta_P F)).$$
\section{About computation of intersection number~: classical results and extensions}\label{rappels}
As in \cite{Fisher}, we write 
${\mathbb C}[[x]]$ the ring of formal power series and ${\mathbb C}[[x^*]]:=
    \bigcup_{N\ge 1}{\mathbb C}[[x^{\frac 1N}]]$ the ring of formal fractional power series.
\begin{defi}
Let $g\in {\mathbb C}[[x^*]]$.
\begin{itemize}
\item If $g(x)=ax^q$ with $a\in\mathbb C^*$ and $q\in\mathbb Q_+$, 
we say that the {\bf degree} of $g$ is equal to $q$.
\item We denote by $LM(g)$ the {\bf lowest degree monomial term} of $g$ and we call
(rational) {\bf valuation} of $g$, 
also denoted $\textrm{val}(g)$ or $\textrm{val}_x(g(x))$ the degree of LM(g).
\end{itemize}
\end{defi}
\begin{defi}
Let $F\in{\mathbb C}[X,Y,Z]$ be a homogeneous polynomial such that $F(0,0,1)=0$.
The {\bf tangent cone} of $V(F)$ at $[0:0:1]$ is $V(LM(F))$ where $LM(F)$ is the sum of
terms of lowest
degree in $F(x,y,1)$. 

We recall that the degree in $\{x,y\}$ of $LM(F)$ is the multiplicity of $[0:0:1]$ in $V(F)$.
\end{defi}
We recall now the classical use of the Weierstrass preparation theorem 
combined with the Puiseux expansions (see \cite{Fisher}, pages 107, 137 and 142).
For every $N\in\mathbb N^*$, we set ${\mathbb C} \langle x^{\frac 1N}\rangle$ and 
${\mathbb C} \langle x^{\frac 1N},y\rangle$ the rings of convergent power series
of $x^{\frac 1 N}$ and of $x^{\frac 1 N},y$. Let ${\mathbb C}\langle x^*\rangle:=
    \bigcup_{N\ge 1}{\mathbb C}\langle x^{\frac 1N}\rangle$ and
${\mathbb C}\langle x^*,y\rangle:=
    \bigcup_{N\ge 1}{\mathbb C}\langle x^{\frac 1N},y\rangle$.

Let $F(x,y,z)\in{\mathbb C}[x,y,z]$ be a homogeneous polynomial.
We suppose that $[0:0:1]$ is a point of $V(F)$ with multiplicity $q$.
This means that $F(x,y,1)$ has valuation $q$ in $(x,y)$.
Suppose that $\{X=0\}$ is not contained in the tangent
cone of $V(F)$ at $[0:0:1]$. This implies that $q$ is also the valuation of $F(0,y,1)$ 
in $y$.

Now, the Weierstrass theorem insures the existence of  a unit $U$ of 
${\mathbb C} \langle x,y\rangle$ and of $\Gamma(x,y)\in {\mathbb C}\langle x\rangle[y]$ 
monic, such that $\Gamma(0,y)$ has degree $q$ in $y$ and
$$F(x,y,1)=U(x,y)\Gamma(x,y) $$
There exists an integer $b\ge 1$ and $\Gamma_1(x,y),...,\Gamma_b(x,y)
\in {\mathbb C}\langle x\rangle[y]$ monic irreducible
such that 
$$\Gamma(x,y)=\prod_{\beta=1}^b\Gamma_\beta(x,y).$$
We recall that the ${\mathcal B}_\beta$'s with equation $\Gamma_\beta=0$ on $z=1$
are the {\bf branches} of $V(F)$ at $[0:0:1]$.
The {\bf tangent line ${\mathcal T}_\beta$ to branch ${\mathcal B}_\beta$ at $[0:0:1]$}
is the reduced tangent cone of ${\mathcal B}_\beta$ at $[0:0:1]$ (the 
notion of tangent cone is well defined with the same definition as for polynomials, see
\cite[P. 148]{Fisher}). Line ${\mathcal T}_\beta$ is given by $(\Gamma_\beta)_xX+
(\Gamma_\beta)_yY=0$.

The degree $e_\beta$ in $y$ of $\Gamma_\beta(x,y)$ is called the {\bf multiplicity of 
the branch} $\mathcal B_\beta$.

Thanks to the Puiseux theorem, for every $\beta\in\{1,...,b\}$,
there exists $\varphi_\beta(t)\in{\mathbb C}\langle t\rangle$ such that
$$\Gamma_\beta(x,y)=\prod_{k=1}^{e_\beta}\left(y-\varphi_{\beta,k}(x)\right),\ \ 
     \mbox{with}\ \ 
    \varphi_{\beta,k}(x):=\varphi_\beta\left(e^{\frac {2ik\pi}
    {e_\beta}}x^{\frac 1{e_\beta}}\right)\in\mathbb C\langle x^*\rangle.$$
Of course we have $q=\sum_{\beta=1}^be_\beta$.

We recall that $\left(y=\varphi_\beta\left(e^{\frac {2ik\pi}
    {e_\beta}}x^{\frac 1{e_\beta}}\right),\ k=1,...,e_\beta
  \right)$ are equations of
the {\bf pro-branches} of branch $V(H_\beta)$ (this notion can be found
in \cite{Halphen,Wall}).

This is summarized in the following theorem in which the pro-branches
are numbered by $i\in\mathcal I=\{1,...,q\}$ and are denoted by $g_i$.
\begin{thm}\label{weierstrass-puiseux}
Let $F(x,y,z)\in{\mathbb C}[x,y,z]$ be a homogeneous polynomial.
We suppose that $[0:0:1]$ is a point of $V(F)$ with multiplicity $q$
and  such that $\{x=0\}$ is not contained in the tangent
cone of $V(F)$ at $[0:0:1]$. Then there exists $U(x,y)$ being a unit of 
${\mathbb C} \langle x,y\rangle$ and $g_1,...,g_q\in{\mathbb C}\langle x^*\rangle$
such that
$$F(x,y,1)=U(x,y)\prod_{i=1}^q(y-g_i(x))\ \ \ \mbox{in}\ \ {\mathbb C}\langle x^*,y\rangle. $$
\end{thm}
Now, let us introduce now the notions of intersection numbers for pro-branches and
for branches (see \cite {Halphen,Wall}).
\begin{defi}\label{multiplicite-defi}
Let $F(x,y,z)$ and $G(x,y,z)$ in ${\mathbb C}[x,y,z]$ 
be two homogeneous polynomials. Let $m_1\in V(F)\cap V(G)$
be a point of multiplicity $q$ of $V(F)$.
Let $M\in GL({\mathbb C}^3)$ be such that $ m_1=\Pi(M(0,0,1))$ and such that
$\{x=0\}$ is not contained in the tangent cone of $V(F\circ M)$ at $[0:0:1]$.

According to theorem \ref{weierstrass-puiseux}, we have
$$F(M(x,y,1))=U(x,y)\prod_{i=1}^q(y-g_i(x)),\ \mbox{with}\ 
   g_i(x)\in{\mathbb C}[[x^*]],$$
$U(x,y)$ being a unit in the ring of convergent power series
${\mathbb C}\langle x,y\rangle$. We also use notations $\mathcal B_\beta$ for branches
of $V(F\circ M)$ and $\mathcal T_\beta$
for tangent line to $\mathcal B_\beta$ at $[0:0:1]$.

We define
\begin{itemize}
\item the {\bf intersection numbers $i_{m_1}^{(i,j)}$ of pro-branches} of 
$V(F\circ M)$ of equations $y=g_i(x)$ and $y=g_j(x)$
are given by 
$$i_{m_1}^{(i,j)}:=\textrm{val}(g_i-g_j),$$ 
\item the {\bf tangential intersection number $i_{m_1}^{(i)}$ of pro-branch} of 
$V(F\circ M)$ of equation 
$y=g_i(x)$
are given by the following formula
\footnote{The fact that $X=0$ is not contained in the tangent cone of $V(F\circ M)$
implies that $val g_i\ge 1$.}
$$i_{m_1}^{(i)}=\textrm{val}_x(g_i(x)-g'_i(0)x)$$
($i_{m_1}^{(i)}$ corresponds to the intersection number of the pro-branch
of equation $y=g_i(x)$ with
the tangent line $\mathcal D_{m_1}^{(i)}=\mathcal T_{\beta}$
where $\mathcal B_\beta$ is the branch 
associated to this pro-branch ; this tangent line has equation $y=g_i'(0)x$).
\item the {\bf intersection number $i_{[0:0:1]}(V(G\circ M),\mathcal B_\beta)$ of a branch
$\mathcal B_\beta$ of $V(F\circ M)$ with $V(G\circ M)$} 
is defined by 
$$i_{[0:0:1]}(V(G\circ M),\mathcal B_\beta):=
    \sum_{i\in\mathcal I_\beta}\textrm{val}_x(G(M(x,g_i(x),1))),$$
where $\mathcal I_\beta$ is the set of indices $i\in\{1,...,q\}$ such that 
the pro-branch of $V(F\circ M)$ of equation $y=g_i(x)$ is associated to branch 
$\mathcal B_\beta$.
\end{itemize}
\end{defi}
We recall that, under hypotheses of this definition,
the intersection number $i_{m_1}(V(F),V(G))$ is given by
\begin{equation}\label{calculmultiplicite}
i_{m_1}(V(F),V(G))=\sum_{i=1}^q \textrm{val}_x\left(G(x,g_i(x),1)\right)
     =\sum_{\beta=1}^bi_{[0:0:1]}(V(G\circ M),\mathcal B_\beta). 
\end{equation}
This observation will be crucial here in computations of intersection numbers.
\begin{rqe}\label{indice}
Quantity $i_{m_1}^{(i)}$ corresponds to the degree of the smallest degree term of $g_i$
of degree greater than or equal to  1 (i.e. $g_i(x)=\alpha_1 x+\alpha x^{i_{m_1}^{(i)}}+...$ with $\alpha\ne 0$).
It is not difficult to see that
$$i_{m_1}^{(i)}=\textrm{val}_x(g_i(x)-g'_i(0)x)=\textrm{val} _x(xg'_i(x)-g_i(x))=
\textrm{val}_x(g_i'(0)-g_i'(x))+1=\textrm{val}g_i''+2.$$
\end{rqe}
Another interesting observation is that, in some sense, the notion of branches
as well as their intersection numbers 
do not depend on the choice of matrix $M$. This is the oject of the following proposition,
the proof of which is postponed in appendix \ref{branches}.
\begin{prop}\label{multiplicite}
Let $F(x,y,z)$ in ${\mathbb C}[x,y,z]$ 
be a homogeneous polynomial. Let 
$m_1$ be a point of multiplicity $q$ of $V(F)$.
Let $M,\hat M\in GL({\mathbb C}^3)$ be such that $m_1=\Pi(M(0,0,1))=\Pi(\hat M(0,0,1))$ 
and such that
$\{x=0\}$ is not in the tangent cones of $V(F\circ M)$ or of $V(F\circ \hat M)$ at $[0:0:1]$.
Then
\begin{itemize}
\item the multiplicities of $[0:0:1]$ on $V(F\circ M)$ and on $V(F\circ\hat M)$ 
are equal,
\item $V(F\circ M)$ and $V(F\circ\hat M)$ have the same number $b$ of branches at $[0:0:1]$,
$\mathcal B_1,...,\mathcal B_b$ and $\hat{\mathcal B}_1,...,\hat{\mathcal B}_b$ respectively.
\item There exists a permutation $\sigma$ of $\{1,...,b\}$ such that, for every 
$\beta\in\{1,...,b\}$,
\begin{itemize} 
\item $\mathcal B_{\beta}$ and $\hat {\mathcal B}_{\sigma(\beta)}$ have the same multiplicity 
$e_\beta$,
\item If $\mathcal T_{\beta}$ and $\hat{\mathcal T}_{\sigma(\beta)}$ 
of $\mathcal B_\beta$ and $\hat{\mathcal B}_{\sigma(\beta)}$ at $[0:0:1]$, then
$M(\mathcal T_\beta)=\hat M(\hat{\mathcal T}_{\sigma(\beta)}), $
\item for every homogeneous polynomial $G(x,y,z)\in{\mathbb C}[x,y,z]$, we have
$$i_{[0:0:1]}(V(G\circ M),{\mathcal B}_\beta)=i_{[0:0:1]}(V(G\circ \hat M),\hat{\mathcal B}_
 {\sigma(\beta)})$$
\item In a neighbourhood of $[0:0:1]$, if $\mathcal B_\beta$ has equation 
$\Gamma_\beta(x,y)=0$ on $z=1$, $\hat{\mathcal B}_{\sigma(\beta)}$ has equation 
$\Gamma_\beta\left(\frac{X(x,y,1)}{Z(x,y,1)},\frac{Y(x,y,1)}{Z(x,y,1)}\right)=0$ 
where $X$, $Y$, $Z$ are the coordinates of map $M^{-1}\circ\hat M$.
\end{itemize}
\end{itemize}
\end{prop}
\begin{prop}\label{ptitlemme}
Let $\mathcal C=V(F)$ with $F$ a homogeneous polynomial of degree $d\ge 2$.
Let $m_1$ be a point of $\mathcal C$ of multiplicity $q$ and let 
$P\in\mathbb C^3\setminus\{0\}$ be such that $\Delta_PF(m_1)=0$.
Assume assumptions of Definition \ref{multiplicite-defi}.
We have
$$ i_{m_1}({\mathcal C},V(\Delta_P F))=
   \left[\sum_{i\in\mathcal I}\sum_{j\in\mathcal I:j\ne i} i_{m_1}^{(i,j)}\right]
     + \sum_{i\in\mathcal I} \left[(i_{m_1}^{(i)}-1)
      {\mathbf 1}_{\Pi(P)\in {\mathcal D}_{m_1}^{(i)}} +{\mathbf 1}
         _{\Pi(P)=m_1}\right],$$
and
$$ i_{m_1}({\mathcal C},V(H_F))
    =\left(3\sum_{i\in\mathcal I}\sum_{j\in\mathcal I:j\ne i} i_{m_1}^{(i,j)}\right)
     +\sum_{i\in\mathcal I} [i_{m_1}^{(i)}-2].$$
\end{prop}
We will see in Remark \ref{ptiterqe} that, 
with the notations of Definition \ref{multiplicite-defi}, 
the values of 
$$V_{m_1}:=
\sum_{i\in\mathcal I}\sum_{j\in\mathcal I:j\ne i} i_{m_1}^{(i,j)}\ \ \mbox{and}
   \ \ I_{m_1}:=\sum_{i\in\mathcal I}[i_{m_1}^{(i)}-2]$$
do not depend on the choice of $M$.
\begin{coro}\label{Plucker}
Proposition \ref{ptitlemme} combined with \cite[p. 91-92]{Fisher} (one can also
use our fundamental lemma \ref{lemmefondamental}) can be used to get 
precised Pl\"ucker formulas for the class and for the number of inflection
points for a general plane algebraic curve. Indeed, for generic $P\in\mathbb P^2$, we have 
$$\boxed{d^\vee=d(d-1)-\sum_{m_1\in Sing(\mathcal C)}i_{m_1}(\mathcal C,V(\Delta_P F))
   =d(d-1)-\sum_{m_1\in Sing(\mathcal C)}V_{m_1}} $$
and
$$\boxed{3d(d-2)-\sum_{m_1\in Sing(\mathcal C)}i_{m_1}(\mathcal C,V(H_F))
   =\sum_{m_1\in Reg(\mathcal C)} I_{m_1}}$$
which corresponds to the number of inflection points.
Moreover, we have
$$ \sum_{m_1\in Sing(\mathcal C)}i_{m_1}(\mathcal C,V(H_F))= 3
      \sum_{m_1\in Sing(\mathcal C)}V_{m_1}+\sum_{m_1\in Sing(\mathcal C)}I_{m_1}.$$
\end{coro}

Applying Proposition \ref{ptitlemme} to non-singular 
points (including flexes), nodes and cusps, we obtain directly:
\begin{coro}
Under assumption of proposition \ref{ptitlemme},
\begin{itemize}
\item If $\mathcal C$ is smooth at $m_1$ with 
$i_{m_1}({\mathcal C},{\mathcal T}_{m_1}{\mathcal C})=p$ (for some $p\ge 2$),
then we have
$$i_{m_1}({\mathcal C},V(H_F))=p-2 \ \ \mbox{and}\ \ 
   i_{m_1}({\mathcal C},V(\Delta_P F))=(p-1)+{\mathbf 1}_{m_1=\Pi (P)}.$$
\item If $F$ admits at $m_1$ an ordinary node, we have $q=2$, $b=2$, $e_1=e_2=1$,
$i_{m_1}^{(1)}=i_{m_1}^{(2)}=2$, $i_{m_1}^{(1,2)}=1$, and so
$$i_{m_1}({\mathcal C},V(H_F))=6 \ \ \mbox{and}\ \ 
   i_{m_1}({\mathcal C},V(\Delta_P F))=\left\{
  \begin{array}{cc}
    2&\mbox{if } \Pi(P)\not\in(\mathcal D_{m_1}^{(1)}\cup \mathcal D_{m_1}^{(2)})\\
    3&\mbox{if }\Pi(P)\in(\mathcal D_{m_1}^{(1)}\cup \mathcal D_{m_1}^{(2)})
           \setminus\{m_1\}\\
    5&\mbox{if }\Pi(P)=m_1
   \end{array}\right..$$
\item If $F$ admits at $m_1$ an ordinary cusp, we have $q=2$, $b=1$, $e_1=2$,
$\mathcal D_{m_1}^{(1)}=\mathcal D_{m_1}^{(2)}$,
$i_{m_1}^{(1)}=i_{m_1}^{(2)}=3/2$, $i_{m_1}^{(1,2)}=3/2$, and so
$$i_{m_1}({\mathcal C},V(H_F))=8 \ \ \mbox{and}\ \ 
   i_{m_1}({\mathcal C},V(\Delta_P F))=\left\{
  \begin{array}{cc}
    3&\mbox{if } \Pi(P)\not\in\mathcal D_{m_1}^{(1)}\\
    4&\mbox{if }\Pi(P)\in\mathcal D_{m_1}^{(1)}\setminus\{m_1\}\\
    6&\mbox{if }\Pi(P)=m_1
   \end{array}\right..$$
\end{itemize}
\end{coro}
In this corollary, we recognize the terms appearing in the classical Pl\"ucker formulas
(see \cite[p. 278-279]{Griffiths-Harris}). 
Proposition \ref{ptitlemme} is a direct consequence of (\ref{calculmultiplicite})
and of the following lemma.
\begin{lem}\label{LEMME}
Under assumptions of Proposition \ref{ptitlemme}, for every $i=1,...,q$, using
notations $R_i(x):=U(x,g_i(x))\prod_{j\in\mathcal I:j\ne i}(g_i(x)-g_j(x))$ and
$(x_P,y_P,z_P)=M^{-1}(P)$, we have
$$\Delta_PF(M(x,g_i(x),1))=R_i(x)\left[y_P-x_Pg'_i(0)+x_P(g'_i(0)-g'_i(x))
              +z_P(xg'_i(x)-g_i(x))\right] $$
and
$$\textrm{val}_x(\Delta_PF(M(x,g_i(x),1))=
   \left[\sum_{j\in\mathcal I:j\ne i} i_{m_1}^{(i,j)}\right]
     + (i_{m_1}^{(i)}-1)
      {\mathbf 1}_{\Pi(P)\in {\mathcal D}_{m_1}^{(i)}} +{\mathbf 1}
         _{\Pi(P)=m_1}.$$
Moreover, we have
$$H_F(M(x,g_i(x),1))=(\textrm{det} M)^{-2}(d-1)^2(R_i(x))^3g''_i(x)$$
and
$$\textrm{val}_x(H_F(M(x,g_i(x),1))
    =3\left[\sum_{j\in\mathcal I:j\ne i} i_{m_1}^{(i,j)}\right]
     + (i_{m_1}^{(i)}-2).$$
\end{lem}
\begin{proof}
We have, in ${\mathbb C}[[x^*]][y]$,
$$F(M(x,y,1))=U(x,y)G(x,y)\ \ \mbox{with}\ \ 
  G(x,y):=\prod_{i=1}^q (y-g_i(x)),\ \ \mbox{and}\ \ 
  U(0,0)\ne 0,$$
with $g_i(x)\in{\mathbb C}[[x^*]]$. Let $i=1,...,q$. Let us set $F_i(x):=(x,g_i(x),1)$.
We get\footnote{On $\{z=1\}$, we have $(F\circ M)_x=U_xG+UG_x$, $(F\circ M)_y=U_yG+UG_y$,
$G_x(x,y,1)=-\sum_{i\in\mathcal I}g'_i(x)\prod_{j\in\mathcal I:j\ne i}(y-g_j(x))$, 
$G_y(x,y,1)=\sum_{i\in\mathcal I}\prod_{j\in\mathcal I:j\ne i}(y-g_j(x))$. We conclude by using the fact that
$G(x,g_i(x))=0$.}
$$  (F\circ M)_x(F_i(x))=-U(x,g_i(x))g_i'(x)\prod_{j\in\mathcal I:j\ne i}(g_i(x)-g_j(x)),$$
$$  (F\circ M)_y(F_i(x))=U(x,g_i(x))\prod_{j\in\mathcal I:j\ne i}(g_i(x)-g_j(x)) .$$
\begin{itemize}
\item Since $\Delta_PF(M(x,y,1))=\Delta_{M^{-1}P}(F\circ M)(x,y,1)=W_P(x,y,1)+dF(M(x,y,1))$ 
with $W_P(x,y,z):=(x_Pz-xz_P)(F\circ M)_x+(y_Pz-yz_P)(F\circ M)_y$.
So, we have
\begin{eqnarray*}
\Delta_P F(F_i(x)) &=& W_P(F_i(x))\\
  &=& U(x,g_i(x))\left[
      (y_P-g_i(x)z_P)-(x_P-xz_P)g'_i(x)\right] \prod_{j\in\mathcal I:j\ne i}(g_i(x)-g_j(x)),
\end{eqnarray*}
which gives the first formula.
\begin{itemize}
\item If $\Pi (P)\not\in \mathcal D_{m_1}^{(i)}$, i.e. $y_P-x_Pg'_i(0)\ne 0$, then
\begin{equation}\label{delta1}
LM\left(\Delta_P F(M(F_i(x)))\right)=LM\left(U(0,0)\left[y_P-x_Pg'_i(0)\right]
 \prod_{j\in\mathcal I:j\ne i}(g_i(x)-g_j(x))\right);
\end{equation}
\item if $\Pi (P)\in \mathcal D_{m_1}^{(i)}$, i.e. $y_P-x_Pg'_i(0)= 0$,
quantity $(y_P-g_i(x)z_P)-(x_P-xz_P)g'_i(x)$ can be rewritten
$$x_P(g_i'(0)-g_i'(x))+z_P(g_i'(x)x-g_i(x)). $$
We distinguish now the cases $\Pi (P)\in \mathcal D_{m_1}^{(i)}$ and $\Pi(P)=m_1$.
\item if $\Pi (P)\in \mathcal D_{m_1}^{(i)}\setminus\{m_1\}$, i.e. $y_P=x_Pg'_i(0)$ and $x_P\ne 0$,
then
\begin{equation}\label{delta2}
LM\left(\Delta_P F(M(F_i(x)))\right) =LM\left( U(0,0)\left[x_P(g'_i(0)-g'_i(x))\right]
     \prod_{j\in\mathcal I:j\ne i}(g_i(x)-g_j(x))
   \right).
\end{equation}
\item if $\Pi(P)=m_1$, then
\begin{equation}\label{delta3}
LM\left(\Delta_P F(M(F_i(x)))\right)=LM\left( U(0,0)\left[z_P
      (xg'_i(x)-g_i(x))\right] \prod_{j\in\mathcal I:j\ne i}(g_i(x)-g_j(x))
  \right).
\end{equation}
\end{itemize}
We conclude thanks to Remark \ref{indice}.
\item We have
\begin{eqnarray*}
 {H_F}(M(F_i(x)))&=&(\textrm{det} M)^{-2}H_{F\circ M}(F_i(x))\\
   &=&(\textrm{det} M)^{-2}(d-1)^2h_{F\circ M}(F_i(x))\\
  &=& (\textrm{det} M)^{-2}(d-1)^2U^3(x,g_i(x))h_G(x,g_i(x)),
\end{eqnarray*}
with $h_G:=2G_{x,y}+U_xG_y+U_yG_x+UG_{xy}$.
\footnote{Since, on $\{z=1\}$, $F\circ M=UG$, so $(F\circ M)_x=U_xG+UG_x$, $(F\circ M)_y=U_yG+UG_y$
and $(F\circ M)_{xx}=U_{xx}G+2U_xG_x+UG_{xx}$, $(F\circ M)_{yy}=U_{yy}G+2U_yG_y+UG_{yy}$ and
$(F\circ M)_{xy}=U_{xy}G+U_xG_y+U_yG_x+UG_{xy}$. Now the fact that $h_{F\circ M}(F_i(x))
 =U(x,g_i(x))
 h_G(x,g_i(x))$ comes from $G(x,g_i(x))=0$.}
Now, let us compute $h_G(x,g_i(x))$.
We have\footnote{Indeed we have 
$G_x(x,y)=-\sum_{i\in\mathcal I}g'_i(x)\prod_{j\in\mathcal I:j\ne i}(y-g_j(x))$, 
$G_y(x,y)=\sum_{i\in\mathcal I}\prod_{j\in\mathcal I:j\ne i}(y-g_j(x))$,
$G_{xx}(x,y)=\sum_{i\in\mathcal I}\sum_{j\in\mathcal I:j\ne i}g'_i(x)g'_j(y)\prod_{k\ne i,j}(y-g_k(x))
         -\sum_{i\in\mathcal I}g_i''(x)\prod_{j\in\mathcal I:j\ne i}(y-g_j(x))$,
$G_{yy}(x,y)=\sum_{i\in\mathcal I}\sum_{j\in\mathcal I:j\ne i}\prod_{k\ne i,j}(y-g_k(x))$
and
$G_{xy}(x,y)=-\sum_{i\in\mathcal I}\sum_{j\in\mathcal I:j\ne i}g'_i(x)
   \prod_{k\ne i,j}(y-g_k(x))$.
}
$$G_x(x,g_i(x))=-g'_i(x)\prod_{j\in\mathcal I:j\ne i}(g_i(x)-g_j(x)),\ \  
        G_y(x,g_i(x))=\prod_{j\in\mathcal I:j\ne i}(g_i(x)-g_j(x)),$$
$$G_{xx}(x,g_i(x))=2\sum_{j\in\mathcal I:j\ne i}g'_i(x)g'_j(x)\prod_{k\ne i,j}(g_i(x)-g_k(x))+
        (-g_i''(x))\prod_{k\ne i}(g_i(x)-g_k(x)), $$
$$G_{yy}(x,g_i(x))=2\sum_{j\in\mathcal I:j\ne i}\prod_{k\ne i,j}(g_i(x)-g_k(x)), $$
$$G_{xy}(x,g_i(x))=\sum_{j\in\mathcal I:j\ne i}(-g'_i(x)-g'_j(x))\prod_{k\ne i,j}(g_i(x)-g_k(x)). $$
A direct computation gives
$$h_G(\cdot,g_i(\cdot),1)=\left(\prod_{j\in\mathcal I:j\ne i}(g_i-g_j)\right)^3g_i'',$$
and so the two last results since $\textrm{val}(g_i'')=i_{m_1}-2$.
\end{itemize}
\end{proof}
Let us observe that we always have $3\sum_{j\in\mathcal I:j\ne i}\textrm{val}(g_i-g_j)+i_{m_1}-2\ge 0$
(if $y=g_i(x)$ is a pro-branch of a branch $\mathcal B_\beta$ with $e_\beta=1$, then $i_{m_1}\ge 2$,
otherwise $\mathcal B_\beta$ admits at least another pro-branch $y=g_j(x)$ and 
$\textrm{val}(g_i-g_j)+val(g_i'')\ge 2i_{m_1}^{(i)}-2\ge 0$).

Now, according to Proposition \ref{multiplicite} and to Lemma \ref{LEMME}, we have
\begin{rqe}\label{ptiterqe}
Under hypotheses of Proposition \ref{multiplicite},
the following quantities are equal for $M$ and for $\hat M$:
$$\sum_{i\in{\mathcal I}_\beta}i_{m_1}^{(i)}=i_{m_1}({\mathcal T}_\beta,{\mathcal B}_\beta),$$
$$\sum_{i\in{\mathcal I}_\beta}\sum_{j\in\mathcal I:j\ne i}i_{m_1}^{(i,j)} =
        \frac 13\left[ i_{m_1}(V(H_F),\mathcal B_\beta)-
      i_{m_1}({\mathcal T}_\beta,{\mathcal B}_\beta)+2e_\beta\right]
     =i_{m_1}(\mathcal B_\beta,\Delta_P),$$
where $\mathcal I_\beta$ is the set of indices $i\in\mathcal I=\{1,...,q\}$ such that
$y=g_i(x)$ is a pro-branch of $\mathcal B_\beta$
and for any $P\in{\mathbb P}^2\setminus\mathcal T_\beta$.
\end{rqe}
\section{Proof of Theorems \ref{degrecaustique1} and \ref{degrecaustique2}}\label{preuve}
According to the fundamental lemma and to (\ref{calculmultiplicite}), 
to prove Theorems \ref{degrecaustique1} and \ref{degrecaustique2}, we have to compute
intersection numbers of branches of $V(F\circ M)$ at $[0:0:1]$
with $M^{-1}({\mathcal P}_{\Phi,a})$ for some
suitable $M\in GL(\mathbb C^3)$ and for generic $a\in\mathbb P^2$.
This is the aim of the following result.
\begin{prop}\label{calculs}
We suppose that $\mathcal C=V(F)$ is irreducible with degree $d\ge 2$ and that $\mathcal S\not\in
\{\mathcal I,\mathcal J\}$. 
Let $m_1$ be a point of multiplicity $q$ of $\mathcal C$, that is a base point of $\Phi_{F,S}$.
Let $M\in GL({\mathbb C}^3)$ be such that $\Pi(M(0,0,1))=m_1$, such that the tangent 
cone of $V(F\circ M)$ at $[0:0:1]$ does not contain $X=0$.
We write $\mathcal B_1,...,\mathcal B_b$ the branches of $V(F\circ M)$ at $[0:0:1]$; 
$\mathcal T_1,...,\mathcal T_b$ their respective tangent lines. 
Let also $\mathcal I:=\{1,...,q\}$
and $y=g_i(x); i\in\mathcal I$ 
be the equations of the pro-branches of $V(F\circ M)$ at $[0:0:1]$.

Let $\beta\in\{1,...,b\}$. 
Let $\mathcal I_\beta$ be the set of indices 
$i\in\mathcal I$ such that
$y=g_i(x)$ are 
the equations of the pro-branches ot $V(F\circ M)$ associated to
branch $\mathcal B_\beta$ at $m_1$ (for $i\in\mathcal I_\beta$,
we have $\mathcal D_{m_1}^{(i)}=\mathcal T_\beta$).

Then, for a generic point $a\in\mathbb P^2$, we have
$$i_{[0:0:1]}\left({\mathcal B}_\beta,M^{-1}({\mathcal P}_{\Phi_{F,S},a})\right)=
  \sum_{i\in\mathcal I_\beta}
  \left(\alpha_i+3 \sum_{j\in\mathcal I:j\ne i}\textrm{val}(g_i-g_j)\right).$$

with

\noindent{\bf \textbullet\ Generic cases}
\begin{itemize}
\item[(S1)] $\alpha_i:=i_{m_1}^{(i)}-2$ if $i_{m_1}^{(i)}<2$ and $m_1\not\in(\mathcal I\mathcal S)
    \cup(\mathcal J\mathcal S)$,
\item[(S2)] $\alpha_i:=0$ if $i_{m_1}^{(i)}\ge 2$ and $\mathcal I,\mathcal J,\mathcal S\not\in{\mathcal 
D_{m_1}^{(i)}}$
\item[(S3)] $\alpha_i:=0$ if $i_{m_1}^{(i)}<2$, $m_1 \in(\mathcal I\mathcal S)
    \cup(\mathcal J\mathcal S)$ and $\mathcal I,\mathcal J,\mathcal S\not\in{\mathcal 
D_{m_1}^{(i)}}$.
\end{itemize}
{\bf \textbullet\ If $\mathcal I$ is in $\mathcal D_{m_1}^{(i)}$}
\begin{itemize}
\item[(S4)] $\alpha_i:=0$ if $\mathcal I\in\mathcal D_{m_1}^{(i)}\setminus \{m_1\}$, 
$i_{m_1}^{(i)}\ge 2$ and
$\mathcal J,\mathcal S\not\in \mathcal D_{m_1}^{(i)}$.
\item[(S5)]
$\alpha_i:=2( i_{m_1}^{(i)}-1)$ if 
$\mathcal S\in{\mathcal D}_{m_1}^{(i)}\setminus
\{m_1\}$, $\mathcal I\in{\mathcal D}_{m_1}^{(i)}$, 
$\mathcal J\not\in{\mathcal D}_{m_1}^{(i)}$
and $i_{m_1}^{(i)}\ne 2.$

\noindent
$\alpha_i:=\min(3,\beta_1)$ if 
$\mathcal S\in{\mathcal D}_{m_1}^{(i)}\setminus
\{m_1\}$, $\mathcal I\in{\mathcal D}_{m_1}^{(i)}$, 
$\mathcal J\not\in{\mathcal D}_{m_1}^{(i)}$, $i_{m_1}^{(i)}= 2$ and if $\beta_1$ is 
the degree of the lowest degree
term of $g_i(x)$ of degree strictly larger than 2.
\item[(S6)]
$\alpha_i:= i_{m_1}^{(i)}-2$ if $\mathcal I,\mathcal J\in\mathcal D_{m_1}^{(i)}\setminus \{m_1\}$, 
and $\mathcal S\not\in \mathcal D_{m_1}^{(i)}$.
\item[(S7)]
$\alpha_i:= 3i_{m_1}^{(i)}-3$ if $\mathcal I,\mathcal J,\mathcal S
\in\mathcal D_{m_1}^{(i)}\setminus \{m_1\}$.
\item[(S8)]
$\alpha_i:=0$ if $\mathcal I=m_1$, and $\mathcal J,\mathcal S\not\in \mathcal D_{m_1}^{(i)}$.
\item[(S9)]
$\alpha_i:= i_{m_1}^{(i)}-1$ if $\mathcal I=m_1$, $\mathcal J\in\mathcal D_{m_1}^{(i)}$, 
$\mathcal S\not\in \mathcal D_{m_1}^{(i)}$.
\item[(S10)]
$\alpha_i=3i_{m_1}^{(i)}-3$ if $\mathcal I=m_1$, 
$\mathcal S,\mathcal J\in\mathcal D_{m_1}^{(i)}$, 
\end{itemize}

\noindent{\bf \textbullet\ Other cases when $\mathcal S$ is in $\mathcal D_{m_1}^{(i)}$}
\begin{itemize}
\item[(S11)]
$\alpha_i:= i_{m_1}^{(i)}-2$ if $\mathcal S\in\mathcal D_{m_1}^{(i)}\setminus \{m_1\}$ and
$\mathcal I,\mathcal J\not\in \mathcal D_{m_1}^{(i)}$.
\item[(S12)]
$\alpha_i:=i_{m_1}^{(i)}$ if $\mathcal S=m_1$ and 
$\mathcal I,\mathcal J\not\in \mathcal D_{m_1}^{(i)}$.
\item[(S13)]
$\alpha_i:=2i_{m_1}^{(i)}-1$ if $\mathcal S= m_1$, 
$\mathcal I\in \mathcal D_{m_1}^{(i)}$ and $\mathcal J\not\in \mathcal D_{m_1}^{(i)}$.
\item[(S14)]
$\alpha_i:=3 i_{m_1}^{(i)}-2$
if $\mathcal S=m_1$,
$\mathcal I,\mathcal J\in \mathcal D_{m_1}^{(i)}$ and $i_{m_1}^{(i)}\ne 2$.

\noindent
$\alpha_i:=\min(\beta_2+2,6)$
if $\mathcal S=m_1$,
$\mathcal I,\mathcal J\in \mathcal D_{m_1}^{(i)}$ and $i_{m_1}^{(i)}= 2$
and if $\beta_2$ is 
the degree of the lowest degree
term of $g_i(x)$ of degree $\not\in\{1,2,3\}$.
\end{itemize}
\end{prop}
For symetry reasons, once this will be proven, same formulas will also hold true 
if we exchange $\mathcal I$ and $\mathcal J$.
\begin{proof}[Scheme of the proofs of Theorems \ref{degrecaustique1} and \ref{degrecaustique2}]
Assume assumptions of Theorems \ref{degrecaustique1} or \ref{degrecaustique2} hold true.
According to the fundamental lemma and to (\ref{calculmultiplicite}), we have, for
generic $a\in\mathbb P^2$,
$$\textrm{mdeg}(\Sigma_{\mathcal S}(\mathcal C))=\delta_1.
\textrm{deg}(\Sigma_{\mathcal S}(\mathcal C))=3d(d-1)-\sum_{m_1
     \in Base((\Phi_{F,\mathcal S})_{\vert \mathcal C})} i_{m_1} \left({\mathcal C},
   {\mathcal P}_{\Phi_{F,\mathcal S},a}\right),$$
with $\delta_1$ the degree of the rational curve $\Phi_{F,\mathcal S}$.
With the notations of Proposition \ref{calculs}, let us write, for every $m_1$ as in Proposition
\ref{calculs}, 
$$\alpha(m_1):=\sum_{i\in\mathcal I}\alpha_i\ \ \mbox{and}\ \ 
  V_{m_1}:=\sum_{i,j\in\mathcal I:i\ne j} i_{m_1}^{(i,j)}.$$
According to Proposition \ref{calculs}, we get
$$\textrm{mdeg}(\Sigma_{\mathcal S}(\mathcal C))=3d(d-1)-3\sum_{m_1\in Sing(\mathcal C)}
    V_{m_1} - \sum_{m_1
     \in Base((\Phi_{F,\mathcal S})_{\vert \mathcal C})} \alpha(m_1)$$
Now, using Corollary \ref{Plucker}, we get that
$$\textrm{mdeg}(\Sigma_{\mathcal S}(\mathcal C))=3d^\vee - \sum_{m_1
     \in Base((\Phi_{F,\mathcal S})_{\vert \mathcal C})} \alpha(m_1).$$
We conclude the proofs by using the expressions of $\alpha_i$ given in Proposition \ref{calculs}.
\end{proof}
We will use the following technical lemma concerning changes of coordinates.
For any $A,B,S',P\in\mathbb C^3\setminus\{0\}$ and any homogeneous
polynomial $F\in{\mathbb C}[X,Y,Z]$, we define
\begin{equation}\label{phiAB}
\tilde\Phi^{(A,B)}_{F,S'}=-\frac{2H_F.f_{A,S'}.f_{B,S'}}{(d-1)^2}.Id+
    \Delta_{A}F.\Delta_BF.\Delta_{S'}F. S'
      -(\Delta_{S'}F)^2[\Delta_AF.B+\Delta_BF.A],
\end{equation}
where
$f_{A,B}(C)=\textrm{det}(A|B|C)$, for every $A,B,C\in\mathbb C^3$.
We recall that $V(f_{A,B})$ is the line $(\Pi(A)\Pi(B))$.
We have already observed in (\ref{expressionPhi}) that
$$\tilde\Phi_{F,S}=\tilde\Phi^{( I, J)}_{F,S}.$$
\begin{lem}\label{ptitlemme0}
For any $M$ in $GL({\mathbb C}^3)$, any $A,B,S,P$ in ${\mathbb C}^3$ and any
homogeneous polynomial $F$, we have
$$\tilde\Phi_{F\circ M^{-1},M( S)}^{(M(A),M(B))}
    (M(P)) = M (\tilde \Phi_{F,S}^{(A,B)}(P)). $$
\end{lem}
\begin{proof}
The lemma is a direct consequence of the following facts~:
$$H_{F\circ M^{-1}}(M(P))
       =(\textrm{det} M)^{-2}H_F(P), $$
$$f_{M(A),M(B)}(M(P))=\textrm{det}(M)f_{A,B}(P), $$
$$\Delta_{M(A)}(F\circ M^{-1})(M(P))=\Delta_AF(P). $$
\end{proof}
\begin{proof}[Proof of Proposition \ref{calculs}]
Let $m_1$ be a  point of $\mathcal C$ with multiplicity $q$.
We use notations of Proposition \ref{calculs}, in particular $\mathcal I:=\{1,...,q\}$.
Let us write $M_1:=(0,0,1)$.
Now, proposition \ref{multiplicite} and Remark \ref{ptiterqe} allow us to
consider each branch separately and to adapt our change of variable 
to each of them.

Consider a branch $\mathcal B_\beta$ of $\mathcal C$ at $m_1$. Let $i\in\mathcal I_\beta$.
We suppose that our change of variable is such
that $g_i'(0)=0$ (i.e. $\mathcal T_\beta$ has equation $Y=0$). We define $A:=M^{-1}(I)$,
$B:=M^{-1}(J)$, $S':=M^{-1}(S)$ and $\hat F:=F\circ M$.
According to Lemma \ref{ptitlemme0}, we have
$$\tilde\Phi_{F,S}^{(I,J)}(M(P))=M\tilde\Phi^{(A,B)}(P),\ \ \mbox{with}\ \
      \tilde\Phi^{(A,B)}:=\tilde\Phi_{\hat F,S'}^{(A,B)}.$$
To simplify notations, we write
$$F_i(x):=(x,g_i(x),1)\ \ \mbox{and}\ \ R_i(x)=U(x,g_i(x))\prod_{j\in\mathcal I:j\ne i}(g_i(x)-g_j(x)). $$
We know that, for every $a=[a_1:a_2:a_3]$
$$i_{[0:0:1]}\left({\mathcal B}_\beta,M^{-1}({\mathcal P}_{\Phi_{F,S},a})\right)
    =\sum_{i\in\mathcal I_\beta}
     \textrm{val}\left(\sum_{j=1}
  ^3({}^tM\cdot a)_j\tilde\Phi^{(A,B)}_j\circ F_i\right). $$
We notive that, for a generic $a\in\mathbb P^2$, we have
$$\textrm{val}\left(\sum_{j=1}
  ^3a_j\tilde\Phi^{(A,B)}_j\circ F_i\right)=\min\left(\textrm{val}\left(\tilde
\Phi_j^{(A,B)}\circ F_i\right)\ ,\ 
   j=1,2,3\right).$$
Let us rewrite formula (\ref{phiAB}):
\begin{equation}
\tilde\Phi^{(A,B)}=\psi_1.Id+\psi_2. S'
      -\psi_3.B-\psi_4.A,
\end{equation}
with
$$\psi_1:=-\frac{2H_{\hat F}.f_{A,S'}.f_{B,S'}}{(d-1)^2},\ \ 
\psi_2:=\Delta_A{\hat F}.\Delta_B{\hat F}.\Delta_{S'}{\hat F},$$
$$\psi_3:=(\Delta_{S'}{\hat F})^2\Delta_A{\hat F}\ \ \mbox{and}\ \ \psi_4:=(\Delta_{S'}{\hat F})^2\Delta_B{\hat F} .$$
Therefore, $\textrm{val}\left(\sum_{j=1}
  ^3a_j\tilde\Phi^{(A,B)}_j\circ F_i\right)$
is greater than or equal to 
the minimum of the four following quantities (the computation of which comes directly
from lemma \ref{LEMME}):
$$\textrm{val}\left(\psi_1\circ F_i\right)
     =\left(3\sum_{j\in\mathcal I:j\ne i}i_{m_1}^{(i,j)}\right)
      +(i_{m_1}^{(i)}-2)+{\mathbf 1}_{m_1\in(\mathcal I\mathcal S)}
     +{\mathbf 1}_{m_1\in(\mathcal J\mathcal S)}
    +(i_{m_1}^{(i)}-1)({\mathbf 1}_{\mathcal D_{m_1}^{(i)}=(\mathcal I\mathcal S)}
           + {\mathbf 1}_{\mathcal D_{m_1}^{(i)}=(\mathcal J\mathcal S)}),$$
$$\textrm{val}\left(\psi_2\circ F_i\right)
    = \left(3\sum_{j\in\mathcal I:j\ne i}i_{m_1}^{(i,j)}\right)+(i_{m_1}^{(i)}-1)
    \left({\mathbf 1}_{\mathcal I\in\mathcal D_{m_1}^{(i)}}
    +{\mathbf 1}_{\mathcal J\in\mathcal D_{m_1}^{(i)}}
    +{\mathbf 1}_{\mathcal S\in\mathcal D_{m_1}^{(i)}}\right)
     +{\mathbf 1}_{m_1\in\{\mathcal I,\mathcal J,\mathcal S\}},$$
$$\textrm{val}\left(\psi_3
     \circ F_i\right)    =\left( 3\sum_{j\in\mathcal I:j\ne i}i_{m_1}^{(i,j)}\right)
    +(i_{m_1}^{(i)}-1)
    \left({\mathbf 1}_{\mathcal I\in\mathcal D_{m_1}^{(i)}}
    +2\times{\mathbf 1}_{\mathcal S\in\mathcal D_{m_1}^{(i)}}\right)
     +{\mathbf 1}_{m_1\in\{\mathcal I,\mathcal S\}},$$
$$\textrm{val}\left(\psi_4
     \circ F_i\right)    = \left(3\sum_{j\in\mathcal I:j\ne i}i_{m_1}^{(i,j)}\right)
     +(i_{m_1}^{(i)}-1)
    \left({\mathbf 1}_{\mathcal J\in\mathcal D_{m_1}^{(i)}}
    +2\times{\mathbf 1}_{\mathcal S\in\mathcal D_{m_1}^{(i)}}\right)
     +{\mathbf 1}_{m_1\in\{\mathcal J,\mathcal S\}}.$$
But it is not clear whether or not it is equal to this minimum.
Hence, in some cases, we will need more
than the values of these valuations.
It will be useful to notice that, according to Lemma \ref{LEMME}, we have
\begin{equation}\label{d1}
\Delta_P {\hat F}\circ F_i(x)=R_i(x)  \left[y_P -g_i'(x)x_P+z_P(xg'_i(x)-g_i(x))
    \right];
\end{equation}
\begin{equation}\label{d2}
\Delta_P {\hat F}\circ F_i(x) = R_i(x)
   \left(-g_i'(x)x_P+z_P(xg'_i(x)-g_i(x))\right)\ \mbox{if}\ 
    \Pi(P)\in{\mathcal D}_{m_1}^{(i)};
\end{equation}
\begin{equation}\label{d3}
\Delta_P {\hat F}\circ F_i(x)=z_PR_i(x)(xg'_i(x)-g_i(x))\ \ \mbox{if}\ \ \Pi(P)=m_1;
\end{equation}
\begin{equation}\label{H}
{H_{\hat F}}\circ F_i(x) = 
    (d-1)^2(R_i(x))^3g_i''(x).
\end{equation}
\begin{itemize}
\item[(S1)] \underline{Suppose that $i_{m_1}^{(i)}<2$, that $m_1$ does not belong to lines 
$(\mathcal I\mathcal S)$
and $(\mathcal J\mathcal S)$.}

Thanks to lemma \ref{ptitlemme0} and to formulas (\ref{d1}), (\ref{d2}),
(\ref{d3}), (\ref{H})
of the proof of Proposition \ref{ptitlemme}, we have
$$\textrm{val}\left(-\frac{2H_{\hat F}.f_{A,{S'}}.f_{B,{S'}}}{(d-1)^2}\circ F_i\right)
    =\left(3\sum_{j\in\mathcal I:j\ne i}i_{m_1}^{(i,j)}\right)+(i_{m_1}^{(i)}-2) $$
is strictly less than the valuation of the three others terms.
Therefore  
$$\left(3\sum_{j\in\mathcal I:j\ne i}i_{m_1}^{(i,j)}\right)+(i_{m_1}^{(i)}-2)
    \le  \min\left(\textrm{val}_x\left(\tilde\Phi^{(A,B)}_j\circ F_i\right)\ ,\ 
j=1,2,3\right)$$
and
$$\textrm{val}\left(\tilde\Phi^{(A,B)}_3\circ F_i\right)=
    \left(3\sum_{j\in\mathcal I:j\ne i}i_{m_1}^{(i,j)}\right)+(i_{m_1}^{(i)}-2).$$
So 
$$\min\left(\textrm{val}\left(\tilde\Phi^{(A,B)}_j\circ F_i\right)\ ,\ 
j=1,2,3\right)=\left(3\sum_{j\in\mathcal I:j\ne i}i_{m_1}^{(i,j)}\right)+(i_{m_1}^{(i)}-2).$$ 
\item[(S2)] \underline{Suppose that 
$i_{m_1}^{(i)}\ge 2$ and if $\mathcal I,\mathcal J,\mathcal S\not\in{\mathcal D}_{m_1}
   ^{(i)}$.}

Thanks to lemma \ref{ptitlemme0} and to formulas (\ref{d1}), (\ref{d2}),
(\ref{d3}) and (\ref{H}), we have
$$3\sum_{j\in\mathcal I:j\ne i}i_{m_1}^{(i,j)}\le\min\left(\textrm{val}_x\left(\tilde\Phi^{(A,B)}_j(x,g_i(x),1)\right)\ ,\ 
j=1,2,3\right).$$

Adapting our change of variable, we can suppose that ${S'}=(0,1,0)$ and that $g_i'(0)=0$.
This implies that $y_A\ne 0$ and $y_B\ne 0$.
We have 
$$LM(\Delta_A{\hat F}\circ F_i)=y_ALM(R_i),\ LM(\Delta_B{\hat F}\circ F_i)=y_BLM(R_i),$$
$$    LM(\Delta_{S'}{\hat F}\circ F_i)=LM(R_i)\ \ \mbox{and}\ \ 
\textrm{val}(H_{\hat F}\circ F_i)\ge  3\sum_{j\in\mathcal I:j\ne i}i_{m_1}^{(i,j)}.$$
So 
$$\min\left(\textrm{val}\left(\tilde\Phi^{(A,B)}_j\circ F_i\right)\ ,\ 
   j=1,2,3\right)\ge 3\sum_{j\in\mathcal I:j\ne i}i_{m_1}^{(i,j)}$$
$$LM((\Delta_A{\hat F}\Delta_B{\hat F}\Delta_{S'}{\hat F})\circ F_i)=y_Ay_BLM((R_i)^3) $$
and
$$LM(((\Delta_{S'}{\hat F})^2(\Delta_A{\hat F}y_B+\Delta_B{\hat F}y_A))\circ F_i)=-2y_Ay_BLM((R_i)^3). $$
Therefore
$$LM(\tilde\Phi^{(A,B)}_2\circ F_i)= -y_Ay_BLM((R_i)^3) $$
and finally
$$\min\left(\textrm{val}_x\left(\tilde\Phi^{(A,B)}_j\circ F_i\right)\ ,\ 
j=1,2,3\right)=3\sum_{j\in\mathcal I:j\ne i}i_{m_1}^{(i,j)}.$$ 
\item[(S3)] \underline{We suppose that $i_{m_1}^{(i)}<2$, that $m_1$ is in line 
$(\mathcal I\mathcal S)$
and that $\mathcal I,\mathcal J,\mathcal S\not\in{\mathcal D}_{m_1}
   ^{(i)}$.}

Again, we suppose that ${S'}=(0,1,0)$ and that $g_i'(0)=0$.
Since $i_{m_1}^{(i)}> 1$, we observe that
$$\textrm{val}([H_{\hat F}f_{A,{S'}}f_{B,{S'}}]\circ F_i)> 3\sum_{j\in\mathcal I:j\ne i}i_{m_1}^{(i,j)},$$
$$LM((\Delta_A{\hat F}\Delta_B{\hat F}\Delta_{S'}{\hat F})\circ F_i)=y_Ay_BLM((R_i)^3) $$
$$LM([(\Delta_{S'}{\hat F})^2(\Delta_A{\hat F}y_B+\Delta_B{\hat F}y_A)]\circ F_i)=-2y_Ay_BLM((R_i)^3) .$$
As in the previous point, we get that 
$$\min\left(\textrm{val}\left(\tilde\Phi^{(A,B)}_j\circ F_i\right)\ ,\ 
j=1,2,3\right)=3\sum_{j\in\mathcal I:j\ne i}i_{m_1}^{(i,j)}.$$
\item[(S4)] \underline{Suppose that $\mathcal I\in{\mathcal D}_{m_1}^{(i)}\setminus\{m_1\}$,
$i_{m_1}^{(i)}\ge 2$, $\mathcal J,\mathcal S\not\in{\mathcal D}_{m_1}^{(i)}$.}

Assume that $M_1=(0,0,1)$, ${S'}=(0,1,0)$, that $A=(1,0,0)$.
We have  $g_i'(0)=0$. Using (\ref{d1}) for $B$ and ${S'}$, (\ref{d2}) for $A$, we get
$$\textrm{val}([H_{\hat F}f_{A,{S'}}f_{B,{S'}}]\circ F_i)\ge 3\sum_{j\in\mathcal I:j\ne i}i_{m_1}^{(i,j)},$$
$$\textrm{val}((\Delta_A{\hat F}\Delta_B{\hat F}\Delta_{S'}{\hat F})\circ F_i)=\left(
   3\sum_{j\in\mathcal I:j\ne i}i_{m_1}^{(i,j)}
     \right)+i_{m_1}^{(i)}-1> 3\sum_{j\in\mathcal I:j\ne i}i_{m_1}^{(i,j)},$$
$$\textrm{val}([(\Delta_{S'}{\hat F})^2\Delta_B{\hat F}]\circ F_i)=3\sum_{j\in\mathcal I:j\ne i}i_{m_1}^{(i,j)},$$
$$\textrm{val}([(\Delta_{S'}{\hat F})^2\Delta_A{\hat F}]\circ F_i)=\left(
3\sum_{j\in\mathcal I:j\ne i}i_{m_1}^{(i,j)}\right)+i_{m_1}^{(i)}-1.$$
So
$$\min\left(\textrm{val}_x\left(\tilde\Phi^{(A,B)}_j\circ F_i\right)\ ,\ 
j=1,2,3\right)\ge 3\sum_{j\in\mathcal I:j\ne i}i_{m_1}^{(i,j)}.$$ 
Moreover $\textrm{val}(\tilde\Phi^{(A,B)}_1\circ F_i)=\textrm{val}([(\Delta_{S'}{\hat F})^2\Delta_B{\hat F}]\circ F_i)$.
So
$$\min\left(\textrm{val}_x\left(\tilde\Phi^{(A,B)}_j\circ F_i\right)\ ,\ 
j=1,2,3\right)= 3\sum_{j\in\mathcal I:j\ne i}i_{m_1}^{(i,j)}.$$ 
\item[(S5)] \underline{Suppose that $\mathcal S\in{\mathcal D}_{m_1}^{(i)}\setminus
\{m_1\}$, $\mathcal I\in{\mathcal D}_{m_1}^{(i)}$ 
and $\mathcal J\not\in{\mathcal D}_{m_1}^{(i)}$.}

Assume that $M_1=(0,0,1)$, ${S'}=(1,0,0)$, $B=(0,1,0)$, $g_i'(0)=0$. We have
$y_A=0$, $z_A\ne 0$, $f_{A,{S'}}(x,y,1)=z_Ay-y_A=z_Ay$ and $f_{B,{S'}}(x,y,1)=-1$.
Using (\ref{d1}) for $B$, (\ref{d2}) for $A$ and ${S'}$ and (\ref{H}), we get
$$-\frac{2[H_{\hat F}f_{A,{S'}}f_{B,{S'}}]( F_i(x))}{(d-1)^2}= 
 2z_A(R_i(x))^3g_i''(x)g_i(x),$$
$$\Delta_{S'}{\hat F}(F_i(x))=-R_i(x)g'_i(x),\ \ \Delta_B{\hat F}(F_i(x))=R_i(x),$$
    $$\Delta_A{\hat F}(F_i(x))=R_i(x)[-g'_i(x)x_A+z_A(xg'_i(x)-g_i(x))] $$
and so
$$(\Delta_A{\hat F}\Delta_B{\hat F}\Delta_{S'}{\hat F})(F_i(x))=R_i(x)^3(-g_i'(x))
     [-g'_i(x)x_A+z_A(xg'_i(x)-g_i(x))],$$
$$[(\Delta_{S'}{\hat F})^2\Delta_B{\hat F}]( F_i(x))=R_i(x)^3(g_i'(x))^2,$$
$$[(\Delta_{S'}{\hat F})^2\Delta_A{\hat F}]( F_i(x))=R_i(x)^3(g_i'(x))^2[-g'_i(x)x_A+z_A(xg'_i(x)-g_i(x))].$$
We cannot conclude since three terms have the smallest valuation. We will see that if $i_{m_1}^{(i)}=2$,
the smallest degree terms are cancelled.
The situation here requires some precise estimate.
Therefore we have
$$\tilde\Phi^{(A,B)}_1(F_i(x))=  z_A(R_i(x))^3[2g_i''(x)g_i(x)x-g_i'(x)(xg'_i(x)-g_i(x))], $$
$$\tilde\Phi^{(A,B)}_2(F_i(x))= (R_i(x))^3[2z_Ag_i''(x)(g_i(x))^2-
              (g_i'(x))^2[-g'_i(x)x_A+z_A(xg'_i(x)-g_i(x))]$$
$$\tilde\Phi^{(A,B)}_3(F_i(x))=  z_A(R_i(x))^3[2g_i''(x)g_i(x)-(g_i'(x))^2].$$
We use the fact that there exists $\alpha\ne 0$ and $\beta=i_{m_1}^{(i)}>1$ such that
$$LM(g_i(x))=\alpha x^\beta,\ LM(g'_i(x))=\alpha\beta x^{\beta-1}\ 
    \ \mbox{and}\ \ LM(g''_i(x))=\alpha\beta(\beta-1) x^{\beta-2}.$$
We get that 
$$LM(\tilde\Phi^{(A,B)}_1(F_i(x)))= LM( z_A(R_i(x))^3[
   \alpha^2\beta (\beta-1)x^{2\beta-1}] .$$
So
$$\textrm{val}(\tilde\Phi^{(A,B)}_1\circ F_i)=\left( 3\sum_{j\in\mathcal I:j\ne i}i_{m_1}^{(i,j)}
    \right)+2i_{m_1}^{(i)}-1.$$
Moreover
$$\textrm{val}(\tilde\Phi^{(A,B)}_2\circ F_i)\ge\left( 
   3\sum_{j\in\mathcal I:j\ne i}i_{m_1}^{(i,j)}\right)+3(i_{m_1}^{(i)}-1) .$$
Oberve now that 
$$\textrm{val}(\tilde\Phi^{(A,B)}_3\circ F_i)\ge
\left(3\sum_{j\in\mathcal I:j\ne i}i_{m_1}^{(i,j)}\right)+2i_{m_1}^{(i)}-2$$
and that term of degree $2i_{m_1}^{(i)}-2$ of $2g_i''(x)g_i(x)- (g_i'(x))^2$
is 
$$\alpha^2(2\beta(\beta-1)-\beta^2)x^{2\beta-2}
     =\alpha^2\beta(\beta-2)x^{2\beta-2}.$$
Therefore
$$\boxed{\min\left(\textrm{val}_x\left(\tilde\Phi^{(A,B)}_j\circ F_i\right)\ ,\ 
j=1,2,3\right)= \left(3\sum_{j\in\mathcal I:j\ne i}i_{m_1}^{(i,j)}\right)+2i_{m_1}^{(i)}-2\ \ \ \mbox{if}\ \ 
i_{m_1}^{(i)}\ne 2}.$$
Now, if $\beta=2$ and if $LM(g_i(x)-\alpha x^2)=\alpha_1x^{\beta_1}$ (with $\alpha_1\ne 0$
and $\beta_1>2$), we get that
$$val(\tilde\Phi^{(A,B)}_1\circ F_i)= \left(3\sum_{j\in\mathcal I:j\ne i}i_{m_1}^{(i,j)}\right)+3,\ \ 
   val(\tilde\Phi^{(A,B)}_2\circ F_i)\ge \left(3\sum_{j\in\mathcal I:j\ne i}i_{m_1}^{(i,j)}\right)+3$$
and
$$LM(2g_i''(x)g_i(x)- (g_i'(x))^2)=2\alpha\alpha_1(\beta_1-1)(\beta_1-2)x^{\beta_1} $$
and so
$$ \boxed{\min\left(\textrm{val}_x\left(\tilde\Phi^{(A,B)}_j\circ F_i\right)\ ,\ 
j=1,2,3\right)= 3\left(\sum_{j\in\mathcal I:j\ne i}i_{m_1}^{(i,j)}\right)+\min(3,\beta_1)
\ \ \ \mbox{if}\ \ i_{m_1}^{(i)}= 2}.$$

\item[(S6)] \underline{Suppose that
$\mathcal I,\mathcal J\in\mathcal D_{m_1}^{(i)}\setminus \{m_1\}$ and that
$\mathcal S\not\in \mathcal D_{m_1}^{(i)}$.}

We suppose that $M_1=(0,0,1)$, ${S'}=(0,1,0)$, $A=(1,0,0)$. We have $g_i'(0)=0$,
$y_B=0$, $x_B\ne 0$, $z_B\ne 0$, $f_{B,{S'}}(x,y,1)=x_B-xz_B$ and $f_{A,{S'}}(x,y,1)=1$. 

We have
$$LM\left(-\frac{2[H_{\hat F}f_{A,{S'}}f_{B,{S'}}]( F_i(x))}{(d-1)^2}\right)= 
 LM\left(-2x_B(R_i(x))^3g_i''(x)\right),$$
$$LM(\Delta_A{\hat F}\Delta_B{\hat F}\Delta_{S'}{\hat F})(F_i(x)))=LM(R_i(x)^3x_B(g_i'(x))^2),$$
$$LM([(\Delta_{S'}{\hat F})^2\Delta_B{\hat F}]( F_i(x)))=LM(R_i(x)^3(-g_i'(x))x_B),$$
$$LM([(\Delta_{S'}{\hat F})^2\Delta_A{\hat F}]( F_i(x)))=LM(R_i(x)^3(-g_i'(x))).$$
Hence
$$\min\left(\textrm{val}_x\left(\tilde\Phi^{(A,B)}_j\circ F_i\right)\ ,\ 
j=1,2,3\right)\ge \left(3\sum_{j\in\mathcal I:j\ne i}i_{m_1}^{(i,j)}\right)+i_{m_1}^{(i)}-2.$$
Moreover
$$LM(\tilde\Phi^{(A,B)}_3(F_i(x)))=LM(-2(R_i(x))^3g''_i(x)x_B) .$$
Therefore 
$$\min\left(\textrm{val}_x\left(\tilde\Phi^{(A,B)}_j\circ F_i\right)\ ,\ 
j=1,2,3\right)\ge \left(3\sum_{j\in\mathcal I:j\ne i}i_{m_1}^{(i,j)}\right)+i_{m_1}^{(i)}-2.$$
\item[(S7)] \underline{Suppose that $\mathcal I,\mathcal J,\mathcal S
\in\mathcal D_{m_1}^{(i)}\setminus \{m_1\}$.}

We suppose that $M_1=(0,0,1)$, ${S'}=(1,0,0)$. We have $g_i'(0)=0$,
$y_A=y_B=0$, $x_A\ne 0$, $z_A\ne 0$, $x_B\ne 0$, $z_B\ne 0$, $f_{A,{S'}}(x,y,1)=z_Ay$
$f_{B,{S'}}(x,y,1)=z_By$.

We have
$$-\frac{2[H_{\hat F}f_{A,{S'}}f_{B,{S'}}]( F_i(x))}{(d-1)^2}= 
 -2z_Az_B(R_i(x))^3g_i''(x)(g_i(x))^2,$$
$$LM(\Delta_A{\hat F}\Delta_B{\hat F}\Delta_{S'}{\hat F})(F_i(x)))=LM(R_i(x)^3x_Ax_B(-g_i'(x))^3),$$
$$LM([(\Delta_{S'}{\hat F})^2\Delta_B{\hat F}]( F_i(x)))=LM(R_i(x)^3(-g_i'(x))^3x_B),$$
$$LM([(\Delta_{S'}{\hat F})^2\Delta_A{\hat F}]( F_i(x)))=LM(R_i(x)^3(-g_i'(x))^3x_A).$$
Hence
$$\min\left(\textrm{val}_x\left(\tilde\Phi^{(A,B)}_j\circ F_i\right)\ ,\ 
j=1,2,3\right)\ge \left(3\sum_{j\in\mathcal I:j\ne i}i_{m_1}^{(i,j)}\right)+3(i_{m_1}^{(i)}-1).$$
Since
$$\textrm{val}\left(\tilde\Phi^{(A,B)}_1\circ F_i\right)=
      \textrm{val}\left((R_i(x))^3(g'_i(x))^3x_Ax_B\right);$$
Hence
$$\min\left(\textrm{val}_x\left(\tilde\Phi^{(A,B)}_j\circ F_i\right)\ ,\ 
j=1,2,3\right)= \left(3\sum_{j\in\mathcal I:j\ne i}i_{m_1}^{(i,j)}\right)+3(i_{m_1}^{(i)}-1).$$

\item[(S8)] \underline{Suppose that $\mathcal I=m_1$ and 
$\mathcal J,\mathcal S\not\in \mathcal D_{m_1}^{(i)}$.}

We suppose that $M_1=A=(0,0,1)$, ${S'}=(0,1,0)$. We have $g_i'(0)=0$,
$y_B\ne 0$, $f_{B,{S'}}(x,y,1)=x_B-xz_B$, $f_{A,{S'}}(x,y,1)=x_A-xz_A=-x$.

We have
$$-\frac{2[H_{\hat F}f_{A,{S'}}f_{B,{S'}}]( F_i(x))}{(d-1)^2}= 
 2(R_i(x))^3g_i''(x)x(x_B-xz_B),$$
$$LM(\Delta_A{\hat F}\Delta_B{\hat F}\Delta_{S'}{\hat F})(F_i(x)))=LM(R_i(x)^3y_B(xg_i'(x)-g_i(x))),$$
$$LM([(\Delta_{S'}{\hat F})^2\Delta_B{\hat F}]( F_i(x)))=LM(R_i(x)^3y_B),$$
$$LM([(\Delta_{S'}{\hat F})^2\Delta_A{\hat F}]( F_i(x)))=LM(R_i(x)^3(xg_i'(x)-g_i(x))).$$
Hence
$$\min\left(\textrm{val}_x\left(\tilde\Phi^{(A,B)}_j\circ F_i\right)\ ,\ 
j=1,2,3\right)\ge 3\sum_{j\in\mathcal I:j\ne i}i_{m_1}^{(i,j)}.$$
Since
$$\textrm{val}\left(\tilde\Phi^{(A,B)}_3\circ F_i\right)= 
   \textrm{val}\left([(\Delta_{S'}{\hat F})^2\Delta_B{\hat F}]\circ F_i\right),$$
we have
$$\min\left(\textrm{val}_x\left(\tilde\Phi^{(A,B)}_j\circ F_i\right)\ ,\ 
j=1,2,3\right)= 3\sum_{j\in\mathcal I:j\ne i}i_{m_1}^{(i,j)}.$$

\item[(S9)]  \underline{Suppose that $\mathcal I=m_1$, $\mathcal J\in\mathcal D_{m_1}^{(i)}$, 
$\mathcal S\not\in \mathcal D_{m_1}^{(i)}$.}

We suppose that $M_1=A=(0,0,1)$, ${S'}=(0,1,0)$, $B=(1,0,0)$. We have $g_i'(0)=0$,
$f_{B,{S'}}(x,y,1)=x_B-xz_B=x_B$, $f_{A,{S'}}(x,y,1)=x_A-xz_A=-x$.

We have
$$-\frac{2[H_{\hat F}f_{A,{S'}}f_{B,{S'}}]( F_i(x))}{(d-1)^2}= 
 2x_B(R_i(x))^3g_i''(x)x,$$
$$LM(\Delta_A{\hat F}\Delta_B{\hat F}\Delta_{S'}{\hat F})(F_i(x)))=LM(R_i(x)^3(xg_i'(x)-g_i(x))x_B(-g'_i(x))),$$
$$LM([(\Delta_{S'}{\hat F})^2\Delta_B{\hat F}]( F_i(x)))=LM(R_i(x)^3(-g'_i(x))x_B),$$
$$LM([(\Delta_{S'}{\hat F})^2\Delta_A{\hat F}]( F_i(x)))=LM(R_i(x)^3(xg_i'(x)-g_i(x))).$$
Hence
$$\min\left(\textrm{val}_x\left(\tilde\Phi^{(A,B)}_j\circ F_i\right)\ ,\ 
j=1,2,3\right)\ge \left(3\sum_{j\in\mathcal I:j\ne i}i_{m_1}^{(i,j)}\right)+i_{m_1}^{(i)}-1.$$
Observe that
$$\textrm{val}\left(\tilde\Phi^{(A,B)}_3\circ F_i\right)= 
   \textrm{val}_x\left((R_i(x))^3(2xg''_i(x)+g_i'(x))x_B\right).$$
Moreover, if $LM(g_i)=\alpha x^\beta$ for some $\alpha\ne 0$ and some $\beta=i_{m_1}^{(i)}>1$,
we get that $LM(2xg''_i(x)+g_i'(x))=\alpha\beta(2\beta-1)x^{\beta-1}$.
Therefore
$$\min\left(\textrm{val}_x\left(\tilde\Phi^{(A,B)}_j\circ F_i\right)\ ,\ 
j=1,2,3\right)= \left(3\sum_{j\in\mathcal I:j\ne i}i_{m_1}^{(i,j)}\right)+i_{m_1}^{(i)}-1.$$

\item[(S10)] \underline{Suppose that $\mathcal I=m_1$, 
$\mathcal S,\mathcal J\in\mathcal D_{m_1}^{(i)}$.}

We suppose that $M_1=A=(0,0,1)$, ${S'}=(1,0,0)$. We have $g_i'(0)=0$,
$y_B=0$, $x_B\ne 0$, $z_B\ne 0$, 
$f_{B,{S'}}(x,y,1)=z_By-y_B=z_By$, $f_{A,{S'}}(x,y,1)=z_Ay-y_A=y$.

We have
$$-\frac{2[H_{\hat F}f_{A,{S'}}f_{B,{S'}}]( F_i(x))}{(d-1)^2}= 
 -2(R_i(x))^3g_i''(x)(g_i(x))^2z_B,$$
$$LM((\Delta_A{\hat F}\Delta_B{\hat F}\Delta_{S'}{\hat F})(F_i(x)))=LM(R_i(x)^3x_B(xg'_i(x)-g_i(x))(-g'_i(x))^2),$$
$$LM([(\Delta_{S'}{\hat F})^2\Delta_B{\hat F}]( F_i(x)))=LM(R_i(x)^3x_B(-g'_i(x))^3),$$
$$LM([(\Delta_{S'}{\hat F})^2\Delta_A{\hat F}]( F_i(x)))=LM(R_i(x)^3(-g'_i(x))^2(xg'_i(x)-g_i(x)).$$
Hence
$$\min\left(\textrm{val}_x\left(\tilde\Phi^{(A,B)}_j\circ F_i\right)\ ,\ 
j=1,2,3\right)\ge\left( 3\sum_{j\in\mathcal I:j\ne i}i_{m_1}^{(i,j)}\right)+3i_{m_1}^{(i)}-3.$$
Moreover
$$LM\left(\tilde\Phi^{(A,B)}_3( F_i(x))\right)= 
   LM(-[(\Delta_{S'}{\hat F})^2\Delta_B{\hat F}]( F_i(x))),$$
we have
$$\min\left(\textrm{val}_x\left(\tilde\Phi^{(A,B)}_j\circ F_i\right)\ ,\ 
j=1,2,3\right)= \left(3\sum_{j\in\mathcal I:j\ne i}i_{m_1}^{(i,j)}\right)+3i_{m_1}^{(i)}-3.$$

\item[(S11)] \underline{Suppose that $\mathcal S\in{\mathcal D}_{m_1}^{(i)}\setminus\{m_1\}$,
 that $\mathcal I,\mathcal J\not\in{\mathcal D}_{m_1}^{(i)}$.}

We suppose that  ${S'}=(1,0,0)$ and that $g_i'(0)=0$. We have $y_A\ne 0$ and $y_B\ne 0$,
$f_{A,{S'}}(M_1)\ne 0$ and $f_{B,{S'}}(M_1)\ne 0$. Thanks to (\ref{d1}) for $\Delta_A{\hat F}$
and $\Delta_B{\hat F}$, (\ref{d2}) for $\Delta_{S'}{\hat F}$ and 
(\ref{H}), we have
$$\textrm{val}([H_{\hat F}f_{A,{S'}}f_{B,{S'}}]\circ F_i)=\left(
 3\sum_{j\in\mathcal I:j\ne i}i_{m_1}^{(i,j)}\right)+i_{m_1}^{(i)}-2,$$
$$LM([\Delta_A{\hat F}\Delta_B{\hat F}\Delta_{S'}{\hat F}]\circ F_i(x))=y_Ay_BLM((R_i)^3(-g'_i(x))) ,$$
$$LM(((\Delta_{S'}{\hat F})^2\Delta_A{\hat F})\circ F_i(x))=y_ALM((R_i)^3
(-g'_i(x))^2) ,$$
$$LM(((\Delta_{S'}{\hat F})^2\Delta_B{\hat F})\circ F_i(x))=y_BLM((R_i)^3
(-g'_i(x))^2) .$$
Hence,
$$\min\left(\textrm{val}\left(\tilde\Phi^{(A,B)}_j\circ F_i\right)\ ,\ 
j=1,2,3\right)\ge\left( 3\sum_{j\in\mathcal I:j\ne i}i_{m_1}^{(i,j)}\right)+i_{m_1}^{(i)}-2$$
and
$$\textrm{val}(\tilde\Phi^{(A,B)}_3\circ F_i)= \textrm{val}([H_{\hat F}f_{A,{S'}}f_{B,{S'}}]\circ F_i).$$ 
So
$$\min\left(\textrm{val}_x\left(\tilde\Phi^{(A,B)}_j\circ F_i\right)\ ,\ 
j=1,2,3\right)= \left(3\sum_{j\in\mathcal I:j\ne i}i_{m_1}^{(i,j)}\right)+i_{m_1}^{(i)}-2.$$

\item[(S12)] \underline{Suppose that $\mathcal S=m_1$ but that $\mathcal I$
and $\mathcal J$ do not belong to ${\mathcal D}_{m_1}^{(i)}$.}

We suppose that  $g_i'(0)=0$, that ${S'}=(0,0,1)$, $y_A\ne 0$ and $y_B\ne 0$,
$f_{A,{S'}}(x,y,1)=y_Ax-yx_A$ and $f_{B,{S'}}(x,y,1)=y_Bx-yx_B$. 
Thanks to (\ref{d1}) for $\Delta_A{\hat F}$
and $\Delta_B{\hat F}$, (\ref{d3}) for $\Delta_{S'}{\hat F}$ and (\ref{H}), we have
$$LM\left(\left[-\frac{2H_{\hat F}f_{A,{S'}}f_{B,{S'}}}{(d-1)^2}\right]\circ F_i(x)
   \right)=-2 LM(R_i(x)^3g_i''(x))y_Ay_Bx^2,$$
$$LM([\Delta_A{\hat F}\Delta_B{\hat F}\Delta_{S'}{\hat F}](F_i(x)))=y_Ay_BLM((R_i)^3(xg_i'(x)-g_i(x))) ,$$
$$LM(((\Delta_{S'}{\hat F})^2\Delta_A{\hat F})(F_i(x))=y_ALM((R_i)^3(xg_i'(x)-g_i(x))^2),$$ 
$$LM(((\Delta_{S'}{\hat F})^2\Delta_B{\hat F})(F_i(x))=y_BLM((R_i)^3(xg_i'(x)-g_i(x))^2).$$

Valuations of the two first terms are in $\left(3\sum_{j\in\mathcal I:j\ne i}i_{m_1}^{(i,j)}
\right)+i_{m_1}^{(i)}$, 
valuations of the two last terms are in $\left(3\sum_{j\in\mathcal I:j\ne i}i_{m_1}^{(i,j)}
\right)+2i_{m_1}^{(i)}$.
Hence
$$\min\left(\textrm{val}\left(\tilde\Phi^{(A,B)}_j\circ F_i\right)\ ,\ 
j=1,2,3\right)\ge \left(3\sum_{j\in\mathcal I:j\ne i}i_{m_1}^{(i,j)}\right)+i_{m_1}^{(i)}.$$
Suppose that $LM(g_i)=\alpha x^\beta$ with $\beta>1$, then $LM(xg_i'(x))=\alpha\beta x^{\beta}$
and $LM(x^2g_i''(x))=\alpha\beta(\beta-1)x^\beta$.
Therefore
$$LM(\tilde\Phi^{(A,B)}_3(F_i(x)))=LM(R_i(x))^3 y_Ay_B\alpha[-2\beta(\beta-1)+\beta-1]x^\beta.$$
So
$$\min\left(\textrm{val}\left(\tilde\Phi^{(A,B)}_j\circ F_i\right)\ ,\ 
j=1,2,3\right)=\left( 3\sum_{j\in\mathcal I:j\ne i}i_{m_1}^{(i,j)}\right)+i_{m_1}^{(i)}.$$

\item[(S13)] \underline{Suppose that $\mathcal S= m_1$, 
$\mathcal I\in \mathcal D_{m_1}^{(i)}$ and $\mathcal J\not\in \mathcal D_{m_1}^{(i)}$.}

We suppose that $M_1={S'}=(0,0,1)$, $A=(1,0,0)$, $B=(0,1,0)$. We have $g_i'(0)=0$,
$f_{B,{S'}}(x,y,1)=y_Bx-yx_B=x$, $f_{A,{S'}}(x,y,1)=y_Ax-yx_A=-y$.

We have
$$-\frac{2[H_{\hat F}f_{A,{S'}}f_{B,{S'}}]( F_i(x))}{(d-1)^2}= 
 -2(R_i(x))^3g_i''(x)(-g_i(x))x,$$
$$LM(\Delta_A{\hat F}\Delta_B{\hat F}\Delta_{S'}{\hat F})(F_i(x)))=LM(R_i(x)^3(-g'_i(x))(xg'_i(x)-g_i(x)),$$
$$LM([(\Delta_{S'}{\hat F})^2\Delta_B{\hat F}]( F_i(x)))=LM(R_i(x)^3(xg'_i(x)-g_i(x))^2),$$
$$LM([(\Delta_{S'}{\hat F})^2\Delta_A{\hat F}]( F_i(x)))=LM(R_i(x)^3(xg'_i(x)-g_i(x))^2(-g'_i(x)).$$
Hence
$$\min\left(\textrm{val}_x\left(\tilde\Phi^{(A,B)}_j\circ F_i\right)\ ,\ 
j=1,2,3\right)\ge\left( 3\sum_{j\in\mathcal I:j\ne i}i_{m_1}^{(i,j)}\right)+2i_{m_1}^{(i)}-1.$$
Moreover, if $LM(g_i(x))=\alpha x^\beta$ with $\alpha\ne 0$ and $\beta=i_{m_1}^{(i)}> 1$, we have
$$LM\left(\tilde\Phi^{(A,B)}_3\circ F_i)\right)= 
      LM((R_i(x))^3)\alpha^2\beta(\beta-1)x^{2\beta-1},$$
and so
$$\min\left(\textrm{val}_x\left(\tilde\Phi^{(A,B)}_j\circ F_i\right)\ ,\ 
j=1,2,3\right)= \left(3\sum_{j\in\mathcal I:j\ne i}i_{m_1}^{(i,j)}\right)+2i_{m_1}^{(i)}-1.$$

\item[(S14)] \underline{Suppose that $\mathcal S=m_1$ and
$\mathcal I,\mathcal J\in \mathcal D_{m_1}^{(i)}$.}

Assume that $M_1={S'}=(0,0,1)$, $A=(1,0,0)$. We have $g_i'(0)=0$,
$y_B=0$, $x_B\ne 0$, $z_B\ne 0$, $f_{A,{S'}}(x,y,1)=y_Ax-yx_A=-y$ and 
$f_{B,{S'}}(x,y,1)=y_Bx-yx_B=-yx_B$.
Thanks to (\ref{H}), we have
$$-\frac{2[H_{\hat F}f_{A,{S'}}f_{B,{S'}}]( F_i(x))}{(d-1)^2}= 
 -2(R_i(x))^3g_i''(x)(g_i(x))^2x_B$$
the valuation of which is equal to
$\left(3\sum_{j\in\mathcal I:j\ne i}i_{m_1}^{(i,j)}\right)+3i_{m_1}^{(i)}-2$.
Using (\ref{d2}) for $A$ and $B$, (\ref{d3}) for ${S'}$, we get 
$$\Delta_{S'}{\hat F}(F_i(x))=R_i(x)(xg'_i(x)-g_i(x)),\ \ 
\Delta_A{\hat F}(F_i(x))=R_i(x)(-g'_i(x)),$$
    $$\Delta_B{\hat F}(F_i(x))=R_i(x)[-g'_i(x)x_B+z_B(xg'_i(x)-g_i(x))] ,$$
and so
$$(\Delta_A{\hat F}\Delta_B{\hat F}\Delta_{S'}{\hat F})(F_i(x))=R_i(x)^3(-g_i'(x))
     (xg'_i(x)-g_i(x))[-g'_i(x)x_B+z_B(xg'_i(x)-g_i(x))],$$
with valuation
$\left(3\sum_{j\in\mathcal I:j\ne i}i_{m_1}^{(i,j)}\right)+3i_{m_1}^{(i)}-2$,
$$[(\Delta_{S'}{\hat F})^2\Delta_B{\hat F}]( F_i(x))=R_i(x)^3(xg'_i(x)-g_i(x))^2
      [-g'_i(x)x_B+z_B(xg'_i(x)-g_i(x))],$$
with valuation
$\left(3\sum_{j\in\mathcal I:j\ne i}i_{m_1}^{(i,j)}\right)+3i_{m_1}^{(i)}-1$
$$[(\Delta_{S'}{\hat F})^2\Delta_A{\hat F}]( F_i(x))=R_i(x)^3(xg'_i(x)-g_i(x))^2(-g'_i(x)),$$
with valuation $\left(3\sum_{j\in\mathcal I:j\ne i}i_{m_1}^{(i,j)}\right)+3i_{m_1}^{(i)}-1$.
We cannot conclude directly. 
Notice that
\begin{equation}
\tilde\Phi^{(A,B)}_1\circ F_i= R_i^3[2[(xg'_i-g_i)^2g'_i-g_i''(g_i)^2x]x_B-z_B(xg'_i-g_i)^3]
\end{equation}
which has valuation greater than or equal to 
$\left(3\sum_{j\in\mathcal I:j\ne i}i_{m_1}^{(i,j)}\right)+3i_{m_1}^{(i)}-1$.
Moreover, we have
\begin{equation}
 \tilde\Phi^{(A,B)}_2(F_i(x))=-2(R_i(x))^3g_i''(x)(g_i(x))^3x_B,
\end{equation}
which has valuation $\left(3\sum_{j\in\mathcal I:j\ne i}i_{m_1}^{(i,j)}\right)+4i_{m_1}^{(i)}-2$.

Let $\alpha\ne 0$ and $\beta=i_{m_1}^{(i)}>1$ be such that
$$LM(g_i(x))=\alpha x^\beta,\ LM(g'_i(x))=\alpha\beta x^{\beta-1}\ 
    \ \mbox{and}\ \ LM(g''_i(x))=\alpha\beta(\beta-1) x^{\beta-2}.$$
We have
$$\tilde\Phi^{(A,B)}_3(F_i(x))=(R_i(x))^3\left[ x_B\{-2g_i''(x)g_i(x)^2+(g_i'(x))^2(xg_i'(x)-g_i(x))\}]
\right] $$
and so its term of order 
$\left(3\sum_{j\in\mathcal I:j\ne i}i_{m_1}^{(i,j)}\right)+3i_{m_1}^{(i)}-2$ is equal to
$$ LM((R_i(x))^3\alpha^3\beta(\beta-1)x_B(\beta-2)x^{3\beta-2}.$$
Therefore
$$\boxed{\min\left(\textrm{val}_x\left(\tilde\Phi^{(A,B)}_j\circ F_i\right)\ ,\ 
j=1,2,3\right)= \left(3\sum_{j\in\mathcal I:j\ne i}i_{m_1}^{(i,j)}\right)+3i_{m_1}^{(i)}-2\ 
\ \ \mbox{if}\ \ 
i_{m_1}^{(i)}\ne 2}.$$

\noindent\textbullet\ 
Now, we suppose that $\beta=2$ and that 
$LM(g_i(x)-\alpha x^2)=\alpha_1x^{\beta_1}$ (with $\alpha_1\ne 0$
and $\beta_1>2$).
If $\beta_1\ne 3$, we get that
$$LM(-2g_i''(x)g_i(x)^2+(g_i'(x))^2(xg_i'(x)-g_i(x)))=
  -2\alpha^2\alpha_1(\beta_1-2)(\beta_1-3)x^{\beta_1+2} $$
and so
$$val(\tilde\Phi^{(A,B)}_3\circ F_i)= \left(3\sum_{j\in\mathcal I:j\ne i}i_{m_1}^{(i,j)}\right)
+\beta_1+2.$$
Moreover, 
$$val(\tilde\Phi^{(A,B)}_2\circ F_i)=\left( 3\sum_{j\in\mathcal I:j\ne i}i_{m_1}^{(i,j)}\right)
   +6.$$
If $\beta_1<3$, we have
$$ LM(\tilde\Phi^{(A,B)}_1\circ F_i)=-2\, LM((R_i(x))^3)
  x_B\alpha^2\alpha_1(\beta_1-2)(\beta_1-4)x^{3+\beta_1};$$
if $\beta_1\ge 3$, we have
$$ val(\tilde\Phi^{(A,B)}_1\circ F_i) \ge \left( 3\sum_{j\in\mathcal I:j\ne i}
i_{m_1}^{(i,j)}\right)   +6.$$
Therefore
$$\boxed{ \min\left(\textrm{val}_x\left(\tilde\Phi^{(A,B)}_j\circ F_i\right)\ ,\ 
j=1,2,3\right)= \left(3\sum_{j\in\mathcal I:j\ne i}i_{m_1}^{(i,j)}\right)+\min(\beta_1+2,6)\
\mbox{if}\ i_{m_1}^{(i)}= 2\ \mbox{and}\ \beta_1\ne 3}.$$
\noindent\textbullet\ 
Now, assume that $i_{m_1}^{(i)}= 2$, that $\beta_1= 3$ and that
$$LM(g_i(x)-\alpha x^2-\alpha_1x^{\beta_1})=\alpha_2x^{\beta_2},$$
with $\alpha_2\ne 0$ and $\beta_2>\beta_1$.

If $\beta_2<4$, we have
$$ LM(\tilde\Phi^{(A,B)}_3\circ F_i)=-2\, LM((R_i(x))^3)
  x_B\alpha^2\alpha_2(\beta_2-2)(\beta_2-3)x^{2+\beta_2};$$
if $\beta_2\ge 4$, we have
$$ val(\tilde\Phi^{(A,B)}_3\circ F_i) \ge \left( 3\sum_{j\in\mathcal I:j\ne i}
i_{m_1}^{(i,j)}\right)   +6.$$
Moreover
$$val(\tilde\Phi^{(A,B)}_2\circ F_i)=\left( 3\sum_{j\in\mathcal I:j\ne i}i_{m_1}^{(i,j)}
    \right)+6\le val(\tilde\Phi^{(A,B)}_1\circ F_i).$$
Finally,
$$\boxed{ \min\left(\textrm{val}_x\left(\tilde\Phi^{(A,B)}_j\circ F_i\right)\ ,\ 
j=1,2,3\right)= \left( 3\sum_{j\in\mathcal I:j\ne i}i_{m_1}^{(i,j)}\right) +\min(\beta_2+2,6)\
\mbox{if}\ i_{m_1}^{(i)}= 2\ \mbox{and}\ \beta_1= 3}.$$

\end{itemize}
\end{proof}
\section{Proof of Corollary \ref{CORO1}}\label{preuvecoro}
We apply Theorem \ref{degrecaustique1}. To simplify notations, 
let us write $i_{m_1}$ for $i_{m_1}(\mathcal C,\mathcal T_{m_1}\mathcal C)$
for every non-singular point $m_1$ of $\mathcal C$.
According to Remark \ref{REMARQUE}, we have
$$3d^\vee-v_1 
   =3d+\sum_{m_1\in Flex(\mathcal C)}(i_{m_1}-2).$$

\noindent\underline{Suppose first that $\mathcal S\not\in\ell_\infty$}.
We have
$$v_2+v'_2=\sum_{m_1\in Flex(\mathcal C)\setminus\{\mathcal S\},\mathcal S\in
      \mathcal T_{m_1}\mathcal C}(i_{m_1}-2)+\sum_{m_1\ne \mathcal S,\mathcal T_{m_1}
       \mathcal C\subseteq(\mathcal I\mathcal S)\cup(\mathcal J\mathcal S)}i_{m_1}+
       \sum_{m_1\in\mathcal C\setminus(Flex(\mathcal C)\cup\{\mathcal S\}) ,\mathcal T_{m_1}
       \mathcal C\subseteq(\mathcal I\mathcal S)\cup(\mathcal J\mathcal S)} 1,$$
$$v_3=i_{\mathcal S}{\mathbf 1}_{\mathcal S\in\mathcal C}+(i_{\mathcal S}-1)
      {\mathbf 1}_{\mathcal S\in\mathcal C,\mathcal T_{\mathcal S}\mathcal C
     \subseteq(\mathcal I\mathcal S)\cup(\mathcal J\mathcal S)}. $$
So
$$v_2+v'_2+v_3= \left(\sum_{m_1\in Flex(\mathcal C),\mathcal S\in
      \mathcal T_{m_1}\mathcal C}(i_{m_1}-2)\right)+t_0+n_0+2 
        \times {\mathbf 1}_{\mathcal S\in\mathcal C} -
  {\mathbf 1}_{\mathcal S\in\mathcal C;\{\mathcal I,\mathcal J\}\cap\mathcal T_{\mathcal S}
        \mathcal C\ne\emptyset}.$$
Finally
$$ v_4= \left(\sum_{m_1\in Flex(\mathcal C),
      \mathcal T_{m_1}\mathcal C=\ell_\infty}(i_{m_1}-2)\right)
   +{\mathbf 1}_{\mathcal I\in\mathcal C,\mathcal T_{\mathcal I}
         \mathcal C=\ell_\infty}
     +{\mathbf 1}_{\mathcal J\in\mathcal C,\mathcal T_{\mathcal J}
         \mathcal C=\ell_\infty}.$$
From which we get the first formula.

\noindent\underline{Suppose now that $\mathcal S\in\ell_\infty$}.
We have
$$v_2+v'_2=\sum_{m_1\in \mathcal C\setminus\{\mathcal S\},\mathcal S\in
      \mathcal T_{m_1}\mathcal C}\left[(i_{m_1}-2) + (2i_{m_1}-1){\mathbf 1}
          _{\mathcal T_{m_1}\mathcal C=\ell_\infty}\right]$$
and
$$v_3=i_{\mathcal S}{\mathbf 1}_{\mathcal S\in\mathcal C}+4\times{\mathbf 1}
          _{\mathcal S\in\mathcal C,\mathcal T_{\mathcal S}\mathcal C=\ell_\infty,
             i_{\mathcal S}=2}+(2i_{S}-2) \times{\mathbf 1}
          _{\mathcal S\in\mathcal C,\mathcal T_{\mathcal S}\mathcal C=\ell_\infty,
             i_{\mathcal S}\ne 2}.$$
Therefore, we have
$$v_2+v'_2+v_3= \sum_{m_1\in \mathcal C,\mathcal S\in
      \mathcal T_{m_1}\mathcal C}\left[(i_{m_1}-2) + (2i_{m_1}-1){\mathbf 1}
          _{\mathcal T_{m_1}\mathcal C=\ell_\infty}\right]
           +3\times {\mathbf 1}_{\mathcal S\in\mathcal C}-
           2\times {\mathbf 1}
          _{\mathcal S\in\mathcal C,\mathcal T_{\mathcal S}\mathcal C=\ell_\infty,
             i_{\mathcal S}\ne 2}.$$
The second result follows.

\appendix
\section{Proof of Proposition \ref{multiplicite}}\label{branches}
Observe that $M^{-1}\hat M$ corresponds to the left-multiplication by
$\left(\begin{array}{ccc}a&b&0\\c&d&0\\e&f&h\end{array}\right)$ for some complex numbers
$a,b,c,d,e,f,h$ such that $ad-bc\ne 0$ and $h\ne 0$.
We define $X(x,y,z):=ax+by$, $Y(x,y,z):=cx+dy$, $Z(x,y,z):=ex+fy+hz$ and
$$\Theta(x,y):=\left(\frac{X(x,y,1)}{Z(x,y,1)},\frac{Y(x,y,1)}{Z(x,y,1)}\right).$$
Let $d_0$ be the degree of polynomial $F$.
By homogeneity of $F$, we have
$$ (F\circ\hat M)(x,y,1))=(h+ex+fy)^{d_0}(F\circ M)(\Theta(x,y),1).$$
Hence, $F\circ M$ and $F\circ\hat M$ have same valuation in $x,y$
so same multiplicity at $[0:0:1]$.

According to the Weierstrass theorem, there exist $U,\hat U$ two units of 
$\mathbb C\langle x,y\rangle$, two integers $b,\hat b$ and
$b+\hat b$ irreducible monic polynomials
$\Gamma_1,...,\Gamma_b,\hat\Gamma_1,...,\hat \Gamma_{\hat b}\in {\mathbb C}\langle x\rangle[y]$ 
with respective degrees in $y$: $e_1,...,e_b,\hat e_1,...,\hat e_{\hat b}$ such that
$$F(M(x,y,1))=U(x,y)\prod_{\beta=1}^b\Gamma_\beta(x,y)\ \ \mbox{and}\ \ 
  F(\hat M(x,y,1))=\hat U(x,y)\prod_{\hat\beta=1}^{\hat b}\hat\Gamma_{\hat\beta}(x,y).$$
Branches $\mathcal B_\beta$ of $V(F\circ M)$ at $[0:0:1]$ 
have equations $\Gamma_\beta=0$ on $z=1$. 
Branches $\hat{\mathcal B}_{\hat\beta}$ of $V(F\circ \hat M)$ at $[0:0:1]$ 
have equations $\hat\Gamma_{\hat\beta}=0$ on $z=1$.
We have
$$ F(\hat M(x,y,1))=(h+ex+fy)^{d_0}U(\Theta(x,y))\prod_{\beta=1}^b\Gamma_\beta(\Theta(x,y)).$$
Let $\beta\in\{1,...,b\}$.
Applying the Weierstass theorem to $\Gamma_\beta(\Theta(x,y))\in{\mathbb C}\langle x,y\rangle$, 
we get the existence of
a unit $U_\beta$ of $\mathbb C\langle x,y\rangle$, an integer $\kappa_\beta\ge 1$
and $\kappa_\beta$ irreducible monic polynomials
$P_{\beta,1},...,P_{\beta, \kappa_\beta}\in {\mathbb C}\langle x\rangle[y]$
such that
$$\Gamma_\beta(\Theta(x,y))=U_\beta(x,y)\prod_{k=1}^{\kappa_\beta}P_{\beta,k}(x,y).$$
Hence, we have
$$F(\hat M(x,y,1))=(h+ex+fy)^{d_0}U(\Theta(x,y))\prod_{\beta=1}^b 
    U_\beta(x,y)\prod_{k=1}^{\kappa_\beta}P_{\beta,k}(x,y).$$
By unicity of factorisation, this implies that every $P_{\beta,k}$ is equal to 
$\hat\Gamma_{\hat\beta}$ for some $\hat\beta$. Hence $b\le\sum_{\beta=1}^b\kappa_\beta=\hat b$.
Doing the same with exchanging the roles of $M$ and of $\hat M$, we finally get that
$b=\hat b$ and $\kappa_\beta=1$ for every $\beta\in\{1,...,b\}$.
This implies the existence of a permutation $\sigma$ of $\{1,...,b\}$ such that
\begin{equation}\label{identite}
\Gamma_\beta(\Theta(x,y)) = U_\beta(x,y)\hat\Gamma_{\sigma(\beta)}(x,y).
\end{equation}
Let us prove that 
$\mathcal T_\beta=M^{-1}\hat M(\hat{\mathcal T}_{\sigma(\beta)}).$
The tangent line $\mathcal T_\beta$ to $\mathcal B_\beta$ at $[0:0:1]$ has equation
$(\Gamma_\beta)_x(0,0)X+(\Gamma_\beta)_y(0,0)Y=0. $
Hence, $\hat M^{-1}( M({\mathcal T}_{\beta}))$ has equation
$$(\Gamma_{\beta})_x(0,0)X(x,y,z)+(\Gamma_{\beta})_y(0,0)Y(x,y,z)=0. $$
Moreover, the tangent line $\hat{\mathcal T}_{\sigma(\beta)}$ to 
$\hat{\mathcal B}_{\sigma(\beta)}$ at $[0:0:1]$ has equation
$(\hat\Gamma_{\sigma(\beta)})_x(0,0)X+(\hat\Gamma_{\sigma(\beta)})_y(0,0)Y=0.$
According to (\ref{identite}), this last equation can be rewritten
$$(U_\beta(0,0))^{-1}
    [(\Gamma_{\beta})_x(0,0)(ax+by)+(\Gamma_{\beta})_y(0,0)(cx+dy)]=0,$$
which is an equation of $\hat M^{-1}( M({\mathcal T}_{\beta}))$.

According to (\ref{identite}), 
$\Gamma_\beta(\Theta(x,y))=0$ is an equation of ${\mathcal B}_{\sigma(\beta)}$.
Hence, $y$-roots $\hat g_{\sigma(\beta),1}(x),...,\hat g_{\sigma(\beta),{ \hat e_\beta}}(x)$ 
of $\hat\Gamma_{\sigma(\beta)}(x,y)$ coincides with $y$-roots
of $\Gamma_\beta(\Theta(x,y))$.

Let us write $e_\beta$ and $\hat e_\beta$ the respective multiplicities of 
$\mathcal B_\beta$ and $\hat{\mathcal B}_{\sigma(\beta)}$.

Since $X=0$ is not in the tangent cone of $V(F\circ \hat M)$ and of $V(F\circ M)$,
we have $(\Gamma_\beta)_y(0,0)\ne 0$, $(\hat\Gamma_{\sigma(\beta)})_y(0,0)\ne 0$,
the function $\hat g_{\sigma(\beta),k}$ is differentiable at $0$ and
$$\hat g_{\sigma(\beta),k}'(0)=-\frac{(\hat\Gamma_{\sigma(\beta)})_x(0,0)}
    {(\hat\Gamma_{\sigma(\beta)})_y(0,0)}.$$
The map $\Theta$ defines a local diffeomorphisms
between two neighbourhood of $(0,0)$. There exists a differentiable function 
$H_{\beta,k}$ such that, for every $(x,y)$ and $(X,Y)$ in a neighourhood of $(0,0)$ satisfying
$(X,Y)=\Theta(x,y)$, we have 
$$y=\hat g_{\sigma(\beta),k}(x)\ \ \Leftrightarrow\ \ Y=H_{\beta,k}(X),
   \ \ \mbox{with}\ \ H_{\beta,k}'(0)=\frac{\hat g'(0)d+c}{\hat g'(0)b+a}
   =-\frac{(\Gamma_{\beta})_x(0,0)}  {(\Gamma_{\beta})_y(0,0)}. $$
We have $\hat g'(0)=\frac{aH_{\beta,k}'(0)-c}{d-bH_{\beta,k}'(0)}$.
Functions $H_{\beta,1},...,H_{\beta,\hat e_\beta}$ are 
$y$-roots of $\Gamma_\beta(x,y)$. Hence we have $\hat e_\beta\le e_\beta$.
For symetry reason, we get that $\hat e_\beta= e_\beta$.
Let $G\in{\mathbb C}[x,y,z]$ be a homogeneous polynomial of degree $d_1$. We have
$$i_{[0:0:1]}(V(G\circ M),\mathcal B_\beta)=\sum_{k=1}^{e_\beta}\textrm{val}_x
   ((G\circ M)(x,H_{\beta,k}(x),1))$$
and

$\displaystyle i_{[0:0:1]}(V(G\circ \hat M),\hat{\mathcal B}_{\sigma(\beta)})
= \sum_{k=1}^{ e_\beta}\textrm{val}_x(G(\hat M(x,\hat g_{\beta,k}(x),1)))$
\begin{eqnarray*}
&=& \sum_{k=1}^{ e_\beta}
   \textrm{val}_x\left(
    (G\circ M)(M^{-1}\hat M(x,\hat g_{\sigma(\beta),k}(x),1))\right)\\
&=& \sum_{k=1}^{ e_\beta}\textrm{val}_x\left((1+ex+f\hat g_{\sigma(\beta),k}(x))^{d_1}
       (G\circ M)(X(x,\hat g_{\sigma(\beta),k}(x)),Y(x,\hat g_{\sigma(\beta),k}(x)),1))
       \right)\\
&=& \sum_{k=1}^{ e_\beta}\textrm{val}_x\left(
       (G\circ M)(X(x,\hat g_{\sigma(\beta),k}(x)),H_{\beta,k}
        (X(x,\hat g_{\sigma(\beta),k}(x),1)))\right).
\end{eqnarray*}
Moreover, since we have
$$LM(X(x,\hat g_{\sigma(\beta),k}(x)))=\alpha x,\ \ \mbox{with}\ \ 
    \alpha:=\frac{a+b\hat g_{\sigma(\beta),k}'(0)}hx
     =\frac{ad-bc}{h(d-bH_{\sigma(\beta),k}'(0))}x,$$
we get
$$ \textrm{val}_x((G\circ M)(X(x,\hat g_{\sigma(\beta),k}(x)),H_{\beta,k}
        (X(x,\hat g_{\sigma(\beta),k}(x),1)))) =
        \textrm{val}_x\left((G\circ M)\left( x, H_{\beta,k}( x),1  \right)  \right).$$

{\bf Acknowledgements~:\/}

The authors thank Jean Marot for stimulating discussions and interesting
references.

\end{document}